\input ./style/arxiv-general.cfg
\documentclass[aop,MSNbibl,dvips]{arximspdf}
\makeatletter
   \@ifpackageloaded{graphicx}{}{\usepackage{graphicx}}
\makeatother

\doi{10.1214/14-AOP973} 
\volume{44}
\issue{1}
\pubyear{2016}
\firstpage{276}
\lastpage{306}
\docsubty{FLA}

\makeatletter

\newcommand{\rrvert}{\vert}
\newcommand{\llvert}{\vert}
\newcommand{\eqref}[1]{(\ref{#1})}
\newcommand{\arctanh}{\operatorname{arctanh}}
\newcommand{\ep}{{\mathbb{E}}}
\newcommand{\pr}{{\mathbb{P}}}
\newcommand{\Z}{{\mathbb{Z}}}
\newcommand{\re}{{\mathbb{R}}}
\newcommand{\wt}[1]{\widetilde{#1}}
\newcommand{\ric}{\operatorname{Ric}}
\newcommand{\vol}{\mathrm{vol}}
\newcommand{\II}{\mathcal{I}}

\newcommand{\PP}{\mathcal{P}}
\newcommand{\AAA}{{\mathcal A}}
\newcommand{\ol}[1]{\overline{#1}}
\newtheorem{theorem}{Theorem}[section]
\newtheorem{theoremm}{Theorem}[section]
\newtheorem{corollary}[theorem]{Corollary}
\newtheorem{corollaryy}[theoremm]{Corollary}
\newtheorem{proposition}[theorem]{Proposition}
\newtheorem{propositionn}[theoremm]{Proposition}
\newtheorem{lemma}[theorem]{Lemma}
\newtheorem{lemmaa}[theoremm]{Lemma}
\newproclaim{definition}[theorem]{Definition}
\newtheorem{conjecture}[theorem]{Conjecture}
\newproclaim{remark}[theorem]{Remark}
\newproclaim{example}[theorem]{Example}
\newtheorem{condition}{Condition}
\makeatother

\begin{document}
\begin{frontmatter}

\title{Discrete versions of the transport equation and the
Shepp--Olkin conjecture\thanksref{T1}}
\runtitle{Discrete transport and the Shepp--Olkin conjecture}
\thankstext{T1}{Supported by an EPSRC grant, Information Geometry of
Graphs, reference EP/I009450/1.}

\begin{aug}
\author[A]{\fnms{Erwan}~\snm{Hillion}\ead[label=e1]{erwan.hillion@univ-amu.fr}}
\and
\author[B]{\fnms{Oliver}~\snm{Johnson}\corref{}\ead[label=e2]{o.johnson@bristol.ac.uk}}
\runauthor{E. Hillion and O.~T. Johnson}
\affiliation{University of Luxembourg and University of Bristol}

\address[A]{Institut Mat\'ematiques de Marseille---UMR 7373\\
Technop\^ole Ch\^ateau-Gombert\\
39, rue Fr\'ed\'eric Joliot-Curie\\
13453 MARSEILLE Cedex 13\\
France\\
\printead{e1}}
\address[B]{School of Mathematics\\
University of Bristol\\
University Walk\\
Bristol BS8 1TW\\
United Kingdom\\
\printead{e2}}
\end{aug}

\received{\smonth{3} \syear{2013}}
\revised{\smonth{9} \syear{2014}}

%
\begin{abstract}
We introduce a framework to consider transport problems for
integer-valued random variables. We introduce
weighting coefficients which allow us to characterize transport
problems in a gradient flow setting,
and form the basis of our introduction of a discrete version of the
Benamou--Brenier formula. Further, we use
these coefficients to state a new form of weighted log-concavity. These
results are applied to prove the monotone case of the Shepp--Olkin
entropy concavity conjecture.
\end{abstract}

%

%
\end{frontmatter}

\section{Introduction}\label{sec1}

In recent years, there has been intensive study of relationships
between entropy and
probabilistic inequalities.
Since the work of
Monge in the eighteenth century, it has been understood how one
probability measure on $\re$ can be smoothly
transformed into another along a path minimizing an appropriate cost.
As described, for example, in
\cite{villani3,villani}, use of the quadratic cost function induces
the quadratic Wasserstein distance $W_2$ which, using
the Benamou--Brenier formula \cite{benamou2,benamou}, can be understood
in terms
of velocity fields arising in gradient models of the kind discussed in
\cite{ambrosio,carlen4,jordan3}.
In such models, concavity
of entropy along the geodesic plays a central role, giving proofs of
inequalities such as HWI, log-Sobolev and transport inequalities; see,
for example, \cite{cordero3}.
A key role is played by log-concavity of the underlying measures and
the Ricci curvature of
the underlying metric space; see, for example, \cite{lott,sturm,sturm2}.

However, this work has almost exclusively focused on continuous random
variables, taking values in $\re^d$,
or more generally on Riemannian
manifolds satisfying a curvature condition. In this paper, we propose a
framework for considering
similar problems for integer-valued random variables. We show how many
natural models of transportation of
discrete random variables can be considered as gradient models and
propose a discrete version of the Benamou--Brenier
formula. As an example of the insights gained by this approach, we
give a proof of a significant new case of the Shepp--Olkin
concavity conjecture \cite{shepp}, which has remained unresolved for
over 30 years.

Shepp and Olkin considered sums of $n$ independent Bernoulli variables
(referred to as Bernoulli sums
throughout this article),
with parameters $p_1,\ldots, p_n$, respectively, where $p_i \in[0,1]$,
and $n$ remains fixed. This sum has
a probability distribution $(f_k)_{k =0, 1, \ldots, n}$, and Shepp and
Olkin conjectured
\cite{shepp} that its entropy is a concave
function of its parameters.

\begin{conjecture}[(\cite{shepp})] \label{conj:SObis}
Consider the entropy of $(f_k)_{k = 0, 1, \ldots, n}$, defined by
\[
H(p_1, \ldots, p_n):= - \sum
_{k=0}^n f_k \log(f_k),
\]
where by convention $0 \log(0) = 0$.
If $p_1, \ldots, p_n \dvtx [0,1] \rightarrow[0,1]$ are affine functions, then
\[
H\dvtx [0,1]\rightarrow\re, \qquad t \mapsto H\bigl(p_1(t),\ldots,
p_n(t)\bigr)
\]
is a concave function in $t$.
\end{conjecture}

We emphasize that Conjecture \ref{conj:SObis}
refers to concavity of entropy in the parameter space. This should be contrasted
with concavity in the space of mass functions themselves. The result
that the entropy of mixtures of mass functions
$f_k(\alpha) := (1-\alpha) f_k^{(1)} + \alpha f_k^{(2)}$ is convex in
$\alpha$ is standard; see, for
example, \cite{cover}, Theorem~2.7.3. Indeed duality between this
parameter representation
and the distribution space is exploited in information geometry (see,
e.g., \cite{amari}), where
the concavity of entropy also plays a central role.

Although the Shepp--Olkin Conjecture \ref{conj:SObis} remains open, we
briefly describe the main cases which
had previously been resolved. First, Shepp and Olkin's original paper
\cite{shepp} showed that the entropy is a Schur-concave
function, stated that Conjecture \ref{conj:SObis} holds for $n=2,3$ and
proved it for interpolation between two binomials, when $p_i(t) = t$
for all $i$.

Second, Theorem~2 of Yu and Johnson \cite{johnsonc6} proves concavity
of the entropy of
$H(T_t X + T_{1-t} Y)$, for $X$ and $Y$ satisfying the ultra
log-concavity property
(see Definition~\ref{def:ulc} below), where $T_t$ represents R\'
{e}nyi's thinning
operation \cite{renyi4}; see equation (\ref{eq:thin}) below. As
remarked in \cite{johnsonc6}, Corollary~1, this resolves
the special case of Conjecture \ref{conj:SObis} where each parameter is
either $p_i(t) = p_i(0) (1-t)$ or
$p_i(t) = p_i(1) t$.

Third, Theorem~1.1 of Hillion \cite{hillion2} resolves the
case where for each $i$, either $p_i(t) \equiv p_i(0)$ for all $t$ or
$p_i(t) = t$ (the translation case of Example~\ref{ex:translation} below).

In this article, given a family of affine functions $p_1(t),\ldots,
p_n(t)$, we consider the associated Bernoulli sum
$(f_k(t))_{k = 0,1, \ldots, n}$ as a function of the spatial variable
$k$ and the time variable $t$. We often omit the explicit dependence of
$f_k$ on $t$. Throughout this article we restrict our
attention to the special case that $p_i' \geq0$ for every $i \in\{
1,\ldots, n\}$, so the random
variables $f_k(t)$ satisfy a stochastic ordering property.
We write the left derivative $ \nabla_1 f_k = f_{k} -f_{k-1}$,
and write $\nabla_2 =  ( \nabla_1  )^2$ for the
map taking $\nabla_2 f_k = f_k - 2 f_{k-1} + f_{k-2}$.

The paper is organized as follows. In Section~\ref{sec:geodesics} we
review properties of
continuous gradient flow models and
develop a framework to prove concavity of entropy.
We introduce and discuss the Benamou--Brenier formula, equation \eqref
{eq:benbren}.

In Section~\ref{sec:discrete} we introduce a formalism to describe
interpolation of discrete probability mass functions $f_k$, motivated
by properties of
the binomial mass functions. In Definition~\ref{def:ourbb},
we propose a discrete analogue of the Benamou--Brenier formula.
A key role is played by our introduction of a family of functions
$\alpha_k(t)$
which are used to generate
mixtures of $f_k$ and $f_{k+1}$. We
write $\AAA$ for the set of measurable functions ${\bolds{\alpha
}}(t) =
(\alpha_0(t), \alpha_1(t), \ldots, \alpha_n(t))$,
where $\alpha_0(t) \equiv0$ and $\alpha_n(t) \equiv1$,
and $0 \leq\alpha_k(t) \leq1$ for all $k$ and $t$.

Our formula of Benamou--Brenier type motivates the following definition:

\begin{definition} \label{def:modtransp}
We say that a family of probability mass functions $f_k(t)$ supported
in $\{0,\ldots, n\}$ is a constant velocity path if for some
$v$ and for some family of probability mass functions $g^{({\bolds
{\alpha}})}_k(t)$
supported in $\{0,\ldots, n-1\}$, it
satisfies a modified transport equation
%
\begin{equation}
\label{eq:modified} \frac{\partial f_k}{\partial t}(t) = - v \nabla_1 \bigl(
g^{({\bolds{\alpha}})}_k(t) \bigr) \qquad\mbox{for $k=0,1, \ldots, n$},
\end{equation}
where for some ${\bolds{\alpha}}(t) \in\AAA$,
%
\begin{equation}
\label{eq:mixture} g^{({\bolds{\alpha}})}_k(t) = \alpha_{k+1}(t)
f_{k+1}(t) + \bigl(1-\alpha_k(t)\bigr) f_k(t)\qquad
\mbox{for $k=0,1, \ldots, n-1$}.
\end{equation}
If $f_k(t)$ satisfies a modified transport equation, we write $h$ for
the function
(not necessarily a probability mass function) satisfying a second-order
modified transport
equation
%
\begin{equation}
\frac{\partial^2 f_k}{\partial t^2}(t) = v^2 \nabla_2 ( h_k )\qquad
\mbox{for $k=0,1, \ldots, n$}. \label{eq:hkdef}
\end{equation}
An explicit expression for $h_k$ is given in equation \eqref
{eq:ModifiedTransport2}.
\end{definition}

In Section~\ref{sec:examples} we discuss the constant velocity path
property. Lemma~\ref{lem:unique}
shows that any such representation is unique, and we give some examples
which lie within this framework.
In Section~\ref{sec:entcon}, by examining the coefficients
${\bolds{\alpha}} \in\AAA$ associated with a constant velocity path
$f_k(t)$, called optimal coefficients for $f_k(t)$,
we give sufficient conditions for the concavity of entropy along the
interpolation.
To be specific, we introduce the following three conditions:

\begin{condition}[($k$-MON)] \label{cond:kmon}
Given $t$, we say the $\alpha_k(t)$ are $k$-monotone at $t$
if
%
\begin{equation}
\label{eq:kmon} \alpha_k(t) \leq\alpha_{k+1}(t)\qquad \mbox{for
all $k = 0, \ldots, n-1$.}
\end{equation}
\end{condition}

\begin{condition}[($t$-MON)] \label{cond:tmon}
Given $t$, we say the $\alpha_k(t)$ are $t$-monotone at $t$
if
%
\begin{equation}
\label{eq:tmon} \frac{\partial\alpha_{k}}{\partial t}(t) \geq0 \qquad\mbox{for all $k = 0, \ldots, n$}.
\end{equation}
\end{condition}

\begin{condition}[(GLC)] \label{cond:glc}
We say probability mass function $f_k$
supported on $\{ 0, 1, \ldots, n \}$ is ${\bolds{\alpha}}$-generalized
log-concave at $t$, denoted ${\mathbf{GLC}}({\bolds{\alpha}
(t)})$, if
for all $k=0, \ldots, n-2$,
%
\begin{equation}
\operatorname{GLC}({\bolds{\alpha}})_k(t) \geq0, \label{eq:glcn}
\end{equation}
where
%
\begin{eqnarray}
\operatorname{GLC}({\bolds{\alpha}})_k(t) &:=& \alpha_{k+1}(t) \bigl(1-
\alpha _{k+1}(t)\bigr) f_{k+1}^2(t)
\nonumber
\\[-8pt]
\\[-8pt]
\nonumber
&&{}-
\alpha_{k+2}(t) \bigl(1-\alpha_k(t)\bigr) f_{k}(t)
f_{k+2}(t).
\end{eqnarray}
\end{condition}

In Section~\ref{sec:entcon} we prove the following theorem:

\begin{theorem}\label{th:ConvexityEntropy} Consider a constant velocity
path $f_k(t)$ of probability mass
functions and associated optimal ${\bolds{\alpha}}(t)$. If
Conditions \ref{cond:kmon}, \ref{cond:tmon} and \ref{cond:glc} hold
for a given $t = t^*$,
then
the entropy $H(f(t))$ is concave in $t$ at $t = t^*$.
\end{theorem}
%

%
\begin{example}
Theorem~\ref{th:ConvexityEntropy} gives a new perspective on Shepp and
Olkin's proof \cite{shepp} that the entropy of
binomial $\operatorname{Bin}(n, t)$ random variables [with probability mass
function $f_k(t)={n\choose k} t^k (1-t)^{n-k}$ for $k = 0,\ldots,n$]
is concave in $t$.
In this case, Example~\ref{ex:binomialint} shows that the optimal
$\alpha_k(t) \equiv k/n$, so
the $k$-monotone and $t$-monotone conditions, Conditions \ref
{cond:kmon} and \ref{cond:tmon} are
clear. Further, $\alpha_k(t) \equiv k/n $ means that GLC Condition
\ref
{cond:glc} reduces
to the ultra log-concavity of order $n$ of Pemantle~\cite{pemantle} and
Liggett \cite{liggett}
(see Definition~\ref{def:ulc}), which clearly holds with equality in
this case.
In fact, in the more general
``symmetric case'' where $p_i'$ does not depend on $i$, Remark~\ref
{rem:symm} shows
$\alpha_k(t) =k/n$, and a similar argument applies.
\end{example}

Finally, we use this formalism to consider the Shepp--Olkin
Conjecture \ref{conj:SObis}, in the ``monotone'' setting $p_i' \geq0$
for all $i$.
In Section~\ref{sec:soint}, we show that in the monotone case,
the Shepp--Olkin interpolation is a constant velocity path in the sense
of Definition~\ref{def:modtransp}.
In Proposition~\ref{prop:alphaMonotonicity} we show that the
$k$-monotone Condition \ref{cond:kmon}
is automatically satisfied in this case. Similarly Proposition~\ref
{prop:GLC} shows that
GLC, Condition \ref{cond:glc}, also holds for all Shepp--Olkin
interpolations in this context.

Unfortunately $t$-monotonicity, Condition \ref{cond:tmon},
does not hold for all (monotone) Shepp--Olkin interpolations.
However, Theorem~\ref{thm:cond1-4} weakens the assumptions of Theorem~\ref{th:ConvexityEntropy}, by
proving that entropy remains concave if
we replace Condition~\ref{cond:tmon} by Condition \ref{cond:delta}, which
is less restrictive, although less transparent in nature.
We complete the proof of
the monotone Shepp--Olkin conjecture by showing in Lemma~\ref
{lem:checkcond4} that Condition \ref{cond:delta}
is satisfied in this case. The proof uses
properties of Bernoulli sum mass functions, including a ``cubic''
inequality Theorem~\ref{th:Symmetric}. In Section~\ref{sec:soconc}, we therefore prove the
main result of this paper:

\begin{theorem}[(Monotone Shepp--Olkin)] \label{thm:SObis}
Consider the entropy of\break $(f_k)_{k = 0, 1, \ldots, n}$, defined by
\[
H(p_1, \ldots, p_n) := - \sum
_{k=0}^n f_k \log(f_k).
\]
If $p_1, \ldots, p_n \dvtx [0,1] \rightarrow[0,1]$ are affine functions with
$p_i' \geq0$ for all $i$, then the function
\[
H\dvtx [0,1]\rightarrow\re ,\qquad  t \mapsto H\bigl(p_1(t),\ldots,
p_n(t)\bigr)
\]
is concave in $t$.
\end{theorem}

\section{Geodesics for continuous random variables} \label{sec:geodesics}
%
\subsection{General framework for concavity of entropy}
In Section~\ref{sec:geodesics} we restate results concerning entropy
and geodesics for random variables on $\re$, using the following
differential equation framework, where the form of (\ref{eq:ghPDE})
motivates equations~(\ref{eq:modified}) and (\ref{eq:hkdef}):

\begin{theorem} \label{th:EntroContinuous}
Let $(f_t(x))_{t \in[0,1]}$ be a smooth family of positive probability
densities on $\re$, such that the entropy $H(t):= -\int_{\re} f_t(x)
\log(f_t(x)) \,dx $ exists for all $t$.
Consider the families of functions $(g_t(x))_{t \in[0,1]}$ and
$(h_t(x))_{t \in[0,1]}$ which satisfy
%
\begin{equation}
\label{eq:ghPDE} \frac{\partial f_t(x)}{\partial t} = - \frac{\partial
g_t(x)}{\partial
x},\qquad \frac{\partial^2 f_t(x)}{\partial t^2} =
\frac{\partial^2 h_t(x)}{\partial x^2}.
\end{equation}
Under technical conditions, such as those listed in Remark~\ref{rem:tech},
the entropy $H(t) $
satisfies
%
\begin{eqnarray}
\label{eq:2derent} H''(t) &=& -\int_{\re}
\biggl( h_t(x) - \frac{g_t(x)^2}{f_t(x)} \biggr) \frac{\partial^2}{\partial x^2} \bigl(
\log\bigl(f_t(x)\bigr) \bigr) \,dx
\nonumber
\\[-8pt]
\\[-8pt]
\nonumber
&&{} - \int_{\re}f_t(x) \biggl(\frac{\partial}{\partial x}
\biggl(\frac
{g_t(x)}{f_t(x)} \biggr) \biggr)^2 \,dx.
\end{eqnarray}
\end{theorem}
\begin{pf} The two conditions listed in part (a) of Remark~\ref
{rem:tech} are those required under Leibniz's rule for differentiation
under the integral sign, yielding
%
\begin{eqnarray}
\label{eq:twoderscont} H''(t) &=& - \int
_{\re} \frac{\partial^2 f_t(x)}{\partial t^2} \log \bigl(f_t(x)\bigr)\,dx
- \int_{\re} \frac{1}{f_t(x)} \biggl(\frac{\partial f_t(x)}{\partial
t}
\biggr)^2 \,dx
\nonumber
\\[-8pt]
\\[-8pt]
\nonumber
&=& -\int_{\re} \frac{\partial^2 h_t(x)}{\partial x^2} \log\bigl(f_t(x)
\bigr)\,dx- \int_{\re} \frac{1}{f_t(x)} \biggl(
\frac{\partial g_t(x)}{\partial
x} \biggr)^2 \,dx.
\end{eqnarray}
Here, by part (b) of Remark~\ref{rem:tech} $g_t(x)$ vanishes at $x =
\pm\infty$, meaning that the
term $\int_{\re} \frac{\partial f_t(x)}{\partial t} \,dx = 0$.
Using the quotient rule
we can write the second term in equation~(\ref{eq:twoderscont}) as
\begin{eqnarray*}
 &&- \int_{\re} \frac{1}{f_t(x)} \biggl(
\frac{\partial
g_t(x)}{\partial x} \biggr)^2 \,dx
\\
&&\qquad =  -\int_\re \biggl( \frac{\partial f_t(x)}{\partial x}
\biggr)^2 \frac{ g_t(x)^2 }{f_t(x)^3}
\\
& &\qquad\quad{} + \frac{\partial f_t(x)}{\partial x} \biggl[ 2 \frac{ g_t(x)
}{f_t(x)} \frac{\partial}{\partial x} \biggl(
\frac{ g_t(x)}{f_t(x)} \biggr) \biggr] + f_t(x) \biggl(
\frac{\partial}{\partial x} \biggl( \frac{
g_t(x)}{f_t(x)} \biggr) \biggr)^2 \,dx
\\
&& \qquad = -\int_\re \biggl( \frac{\partial f_t(x)}{\partial x}
\biggr)^2 \frac{
g_t(x)^2 }{f_t(x)^3}
\\
& &\qquad\quad{} - \frac{\partial^2 f_t(x)}{\partial x^2} \biggl( \frac{
g_t(x)}{f_t(x)} \biggr)^2 \,dx +
f_t(x) \biggl( \frac{\partial}{\partial x} \biggl( \frac{
g_t(x)}{f_t(x)} \biggr)
\biggr)^2 \,dx
\end{eqnarray*}
since we recognize the term in square brackets as a perfect derivative,
and integrate by parts. The remaining conditions
listed in Remark~\ref{rem:tech}(b) justify the necessary integrations
by parts to prove the theorem.
\end{pf}

\begin{remark} \label{rem:tech}
We assume, for example, that the following technical conditions hold:
\begin{longlist}[(a)]
\item[(a)] there exist integrable $\theta_A(x), \theta_B(x)$ such that
for all $t,x$,
$\llvert  \frac{ \partial g_t(x)}{\partial x} (1 + \log(f_t(x)) )
\rrvert  \leq\theta_A(x)$ and
$\llvert  \frac{ \partial^2 h_t(x)}{\partial x^2} \log(f_t(x))
+  ( \frac{ \partial g_t(x)}{\partial x}  )^2 \frac{1}{f_t(x)}
\rrvert  \leq\theta_B(x)$;
\item[(b)]
for each $t \in[0,1]$, functions $g_t(x)$,
$\frac{ \partial h_t(x)}{\partial x} \log(f_t(x))$, $
\frac
{h_t(x)}{f_t(x)} \frac{ \partial f_t(x)}{\partial x} $ and $
( \frac{ g_t(x)}{f_t(x)}  )^2 \frac{ \partial f_t(x)}{\partial x}
$ vanish at $x = \pm\infty$.
\end{longlist}
\end{remark}

Section~\ref{sec:geodesics} aims to motivate results in the case where
all random variables have support on a finite set, so the required
differentiation formulas are automatic. For this reason, we do not
discuss the question of verification of the technical conditions of
Remark~\ref{rem:tech}.

In some sense, an extreme example
for which we can apply Theorem~\ref{th:EntroContinuous} is the following:

\begin{example} \label{ex:conttran}
Consider the translation of probability density $f_0$, where
$f_t(x) := f_0(x-vt)$
for some constant velocity $v>0$. It is then easy to see that $g_t(x) =
v f_t(x)$ and $h_t(x) = v^2 f_t(x)$. Theorem~\ref{th:EntroContinuous}
then confirms shift invariance and makes the entropy $H(t)$ of $f_t$ constant.
\end{example}

\subsection{Benamou--Brenier formula} \label{sec:benbren}

The study of geodesics interpolating between continuous probability
densities exploits properties
of the quadratic Wasserstein distance $W_2$, which (see \cite
{benamou2,benamou,ambrosio})
has a variational characterisation involving velocity fields, given by
the Benamou--Brenier formula (\ref{eq:benbren}) below.

\begin{definition} \label{def:velfield}
Consider fixed smooth distribution functions $F_0$ and $F_1$.
Write $\PP_{\re}(F_0,F_1)$ for the set of probability densities
$f_t(x)$, with corresponding distribution functions
$F_t(x) = \int_{-\infty}^x f_t(y) \,dy$ satisfying constraints $F_t(x)
|_{t=0} = F_0(x)$
and $F_t(x) |_{t=1} = F_1(x)$. Then given any sequence $f_t \in\PP
_{\re
}(F_0,F_1)$, we refer to a function
$v_t$ as a velocity
field if it satisfies
%
\begin{equation}
\label{eq:heateqn} \frac{\partial}{\partial t} f_t(x) + \frac{\partial}{\partial x}
\bigl( v_t(x) f_t(x) \bigr) = 0.
\end{equation}
\end{definition}

Ambrosio, Gigli and Savar\'{e} \cite{ambrosio}, Section~8, give a
careful analysis of conditions under which this type of continuity
equation holds. They consider (see \cite{ambrosio}, Definition~1.1.1)
the class of absolutely continuous curves $\mu_t \in{\mathcal
P}_p(X)$, the set of probability measures with finite $p$th moment on
separable Hilbert space $X$. Theorem~8.3.1 of \cite{ambrosio} shows
that for $p > 1$, a version of equation (\ref{eq:heateqn}) holds for
$\mu_t$ in this class, in fact,
\[
\frac{\partial}{\partial t} \mu_t + \nabla\cdot ( v_t
\mu_t ) = 0,
\]
in the sense of distributions (using the class of smooth cylindrical
test functions).

Further, Theorem~8.3.1 of \cite{ambrosio} shows that under these
conditions the resulting velocity field has $L^p(\mu_t)$ norm dominated
by the metric derivative $| \mu'|(t)$ (as defined in \cite{ambrosio}, equation
(1.1.3)). Using properties of so-called length spaces, this
allows the following formula, first proved by Benamou and Brenier \cite
{benamou2,benamou} for probability measures on $X = \re^d$, to be
recovered for separable Hilbert spaces $X$. For comparison purposes, we
state this Benamou--Brenier formula for the case of $X = \re$:

\begin{theorem}[(\cite{benamou2,benamou})]
Using the notation of Definition~\ref{def:velfield} above, the
quadratic Wasserstein distance is given by
%
\begin{eqnarray}
\label{eq:benbren} W_2^2(F_0,F_1) &
= & \inf_{f_t \in\PP_{\re}(F_0,F_1)} \int_0^1
\biggl( \int_{-\infty}^{\infty} f_t(y)
v_t(y)^2 \,dy \biggr) \,dt
\\
& = & \inf_{f_t \in\PP_{\re}(F_0,F_1)} \int_0^1
\biggl( \int_{-\infty
}^{\infty} \biggl( \frac{\partial F_t}{\partial t}(y)
\biggr)^2 \frac{1}{f_t(y)} \,dy \biggr) \,dt. \label{eq:altbenbren}
\end{eqnarray}
\end{theorem}

\begin{corollary} \label{cor:geod}
If $(f_t)_{t \in[0,1]}$ is a solution to the minimization
problem \eqref{eq:benbren}, then assuming
the technical conditions of Remark~\ref{rem:tech} hold, we can write
%
\begin{equation}
\label{eq:velchange} H''(f_t) = - \int
_{\re} f_t(x) \biggl( \frac{\partial}{\partial x}
v_t(x) \biggr)^2 \,dx \leq0,
\end{equation}
and the inner integral of (\ref{eq:benbren}), $\int_{-\infty
}^{\infty}
f_t(x) v_t(x)^2 \,dx$, is constant in $t$.
\end{corollary}

\begin{pf}
It is shown in \cite{benamou2}, equation (1.14), that if $(f_t)_{t \in
[0,1]}$ is a solution to the minimization problem \eqref{eq:benbren},
then its associated velocity field $v_t(x)$ is, at least formally, a
solution to the equation
%
\begin{equation}
\label{eq:vder} \frac{\partial v_t(x)}{\partial t} = - \frac{\partial
v_t(x)}{\partial
x} v_t(x).
\end{equation}
Taking a further time derivative of \eqref{eq:heateqn} and using
\eqref
{eq:vder},
we deduce a second-order PDE,
%
\begin{eqnarray}
\label{eq:2ndorder} \frac{\partial^2 f_t(x)}{\partial t^2} & = & - \frac{\partial}{\partial x} \biggl(
\frac{\partial v_t}{\partial
t}(x) f_t(x) + v_t(x)
\frac{\partial f_t}{\partial t}(x) \biggr)
\nonumber
\\
& = & \frac{\partial}{\partial x} \biggl( \frac{\partial
v_t(x)}{\partial x} v_t(x)
f_t(x) + v_t(x) \frac{\partial}{\partial x} \bigl(
v_t(x) f_t(x) \bigr) \biggr)
\\
& = & \frac{\partial^2}{\partial x^2} \bigl(v_t(x)^2
f_t(x) \bigr)\nonumber
\end{eqnarray}
(assuming the $t$ and $x$ derivatives can be exchanged).
In the notation of Theorem~\ref{th:EntroContinuous}, we can rewrite equations
(\ref{eq:heateqn}) and (\ref{eq:2ndorder}) in the form $g_t(x) = v_t(x)
f_t(x)$ and $h_t(x) =
v_t(x)^2 f_t(x)$. This makes a clear analogy with the translation case,
Example~\ref{ex:conttran}, and equation (\ref{eq:velchange}) follows
by a straightforward application of Theorem~\ref{th:EntroContinuous}.
Similar calculations using equation (\ref{eq:velchange}) show that $\frac{\partial
}{\partial t}  ( f_t(x) v_t(x)^2  )
= \frac{\partial}{\partial x}  ( f_t(x) v_t(x)^3  )$, and the
result follows.
\end{pf}

This result can be seen as a particular case of results coming from
Sturm--Lott--Villani theory \cite{lott,sturm,sturm2}. This theory
establishes links between the behavior of the entropy functional along
Wasserstein $W_2$-geodesics on a given measured length space and bounds
on the Ricci curvature on this space. In particular, a Riemannian
manifold $(M,g)$ satisfies $\ric\geq0$, where $\ric$ is the Ricci
curvature tensor, if and only if for every absolutely continuous
Wasserstein $W_2$-geodesic $(\mu_t)_{t \in[0,1]}:= (f_t\, d\vol)_{t
\in
[0,1]}$ the entropy function $H(t) := -\int_M f_t \log(f_t) \,d\vol$ is
concave in $t$. This equivalence is used to generalize the definition
of Ricci curvature bounds from the Riemannian framework to the
framework of measured length spaces; that is, metric spaces $(X,d)$ for
which the distance $d(x,y)$ is the infimum of lengths of curves joining
$x$ to $y$.

This theory can be developed to use transportation arguments to prove
probabilistic inequalities involving entropy, such as log-Sobolev,
transport and HWI inequalities. For example, Cordero--Erausquin \cite
{cordero3},
Corollaries~1, 2 and 3, gives simple proofs of these inequalities,
under the condition that relative density $f/\phi_{1/c}$ is log-concave
(in the continuous sense), where $\phi_{1/c}$ is a normal density with
variance $1/c$. This log-concavity condition is known to imply the
Bakry--\'{E}mery condition~\cite{bakry} (see, e.g., \cite
{cordero3,ane}), which is natural in this context. GLC, Condition \ref
{cond:glc}, is introduced as a discrete version of the log-concavity condition.
%
\subsection{Perturbed translations}

\begin{theorem}\label{th:PerturbedTranslation}
Let $(f_t(x))_{t \in[0,1]}$ be a smooth family of positive probability
densities on $\re$ and $(g_t(x))_{t \in[0,1]}$ be defined by
equation \eqref{eq:ghPDE}. If there exists a constant $v$ and a
nondecreasing function $\alpha(t)$ such that
%
\begin{equation}
\label{eq:glinear} g_t(x) = v f_t(x)+\alpha(t)
\frac{\partial}{\partial x} f_t(x),
\end{equation}
then assuming
the technical conditions of Remark~\ref{rem:tech} hold, the entropy
$H(t)$ of $f_t$ is a concave function of $t$.
\end{theorem}

\begin{pf} Using the facts that $\frac{\partial}{\partial t} f_t(x) = -
\frac{\partial}{\partial x} g_t(x)$, and hence
$\frac{\partial^2}{\partial x\, \partial t} f_t(x)= - \frac{\partial
^2}{\partial x^2} g_t(x)$,
we take a derivative of equation (\ref{eq:glinear}) to compute
\begin{eqnarray*}
\frac{\partial}{\partial t} g_t(x) 
&=& - v
\frac{\partial}{\partial x} \biggl( v f_t(x)+\alpha(t) \frac
{\partial}{\partial x}
f_t(x) \biggr)+ \alpha'(t) \frac{\partial
}{\partial x}
f_t(x)
\\
&&{} - \alpha(t) \frac{\partial^2}{\partial x^2} \biggl( v f_t(x)+\alpha(t)
\frac{\partial}{\partial x} f_t(x) \biggr),
\end{eqnarray*}
so the family of functions $h_t(x)$ defined by equation \eqref
{eq:ghPDE} is equal to
%
\begin{equation}
\label{eq:hexpress} h_t(x) = v^2 f_t(x)+2 v
\alpha(t) \frac{\partial}{\partial x} f_t(x) + \alpha(t)^2
\frac{\partial^2}{\partial x^2}f_t(x)-\alpha'(t) f_t(x).
\end{equation}
It is then easy to deduce that
%
\begin{equation}
h_t(x) - \frac{g_t(x)^2}{f_t(x)} = \alpha(t)^2
f_t(x) \biggl(\frac
{\partial^2}{\partial x^2} \log\bigl(f_t(x)\bigr)
\biggr) - \alpha'(t) f_t(x).
\end{equation}
Since in this case $\frac{\partial}{\partial x}  ( g_t(x)/f_t(x)
 ) = \alpha(t) \frac{\partial^2}{\partial x^2}  ( \log f_t(x)
 )^2$,
Theorem~\ref{th:EntroContinuous} gives that
%
\begin{eqnarray}
H''(t) &=& - 2\alpha(t)^2 \int
_{\re} \biggl(\frac{\partial
^2}{\partial
x^2} \log\bigl(f_t(x)
\bigr) \biggr)^2 f_t(x)
\nonumber
\\[-8pt]
\\[-8pt]
\nonumber
&&{}- \alpha'(t) \int
_{\re} \biggl(\frac
{\partial}{\partial x} \log\bigl(f_t(x)
\bigr) \biggr)^2 f_t(x),
\end{eqnarray}
which shows the concavity of $H(t)$.
\end{pf}

Among the consequences of Theorem~\ref{th:PerturbedTranslation} there
are the particular cases where $\alpha(t) \equiv0$, which is the
translation case of Example~\ref{ex:conttran}, and the case where $v =
0$ and $\alpha(t) = -c$ is a constant, in which case $f$ is a solution
to the heat equation $\frac{\partial f_t(x)}{\partial t} = c \frac
{\partial^2 f_t(x)}{\partial x^2}$.\vspace*{1pt}

Theorem~\ref{th:PerturbedTranslation} can be used to study the entropy
of an approximation of a Bernoulli sum by a Gaussian distribution of
the same mean and variance. This motivates the Shepp--Olkin conjecture,
due to the following result:

\begin{theorem} \label{th:SOCont}
Let $p_1,\ldots, p_n \dvtx [0,1] \rightarrow[0,1]$ be affine functions,
and let $\mu(t) := \sum_{i=1}^n p_i(t)$ and $V(t) := \sum_{i=1}^n
p_i(t)(1-p_i(t))$
be the mean and variance functions. Define
%
\begin{equation}
\label{eq:gaussapprox} f_t(x) := \frac{1}{\sqrt{2\pi V(t)}} \exp \biggl(-
\frac{(x- \mu(t))^2}{2
V(t)} \biggr).
\end{equation}
Then the entropy $H(t)$ of $f_t$
is a concave function of $t$.
\end{theorem}

\begin{pf}
Writing $v =\mu'(t)$, since $\mu''(t) = 0$, we can use differential
properties satisfied by Gaussian kernels to compute
\[
g_t(x) = \mu'(t) f_t(x) -
\frac{V'(t)}{2} \frac{\partial
f_t(x)}{\partial x} = v f_t(x) + \alpha(t)
\frac{\partial}{\partial x} f_t(x),
\]
where $\alpha(t) := -\frac{1}{2} V'(t)$. Since we have
$\frac{\partial}{\partial t} \alpha(t)= \sum_{i=1}^n p_i'^2 \geq0$,
we apply Theorem~\ref{th:PerturbedTranslation} to show that the entropy
$H(t)$ of $f_t$ is a concave function of $t$.
The conditions of Remark~\ref{rem:tech} can all be directly verified in
this case; the key is that $g_t(x)/f_t(x)$ is a linear function
of $x$, and $h_t(x)/f_t(x)$ is quadratic in $x$.
This argument works for any Gaussian densities of the form \eqref
{eq:gaussapprox}, where $V''(t) \leq0$, and $\mu(t)$ is an affine
function of $t$.
\end{pf}

\begin{remark}
It is possible to use the explicit expression for the entropy of a
Gaussian random variable to prove Theorem~\ref{th:SOCont}
directly. However, as there is no explicit expression for the entropy
of a sum of Bernoulli variables, it is not possible to adapt such a
proof in the discrete Shepp--Olkin case, and we require the assumption
that all $p_i' \geq0$ in that case.
\end{remark}

\section{Discrete gradient field models} \label{sec:discrete}
%
\subsection{Motivating example and discrete Benamou--Brenier formula}
We now show how natural choices of paths connecting probability mass functions
on the integers can
be viewed in the gradient field framework of Section~\ref{sec:geodesics}. We give a new
perspective on the time
derivative using a series of functions $\alpha_k(t)$, where
$k = 0, \ldots, n$ and $0 \leq t \leq1$. Recall that we use the left
derivative map
$\nabla_1$ defined by
$ \nabla_1 f_k = f_{k} -f_{k-1}$ for any function $f$, and write
$\nabla
_2 =  ( \nabla_1  )^2$ for the
map taking $\nabla_2 f_k = f_k - 2 f_{k-1} + f_{k-2}$. Write $\nabla
_1^*$, defined by
$\nabla_1^* f_k = f_{k} - f_{k+1}$, for its adjoint (with respect to
counting measure).
Recall $\AAA$ denotes the set of measurable
functions ${\bolds{\alpha}}(t) = (\alpha_0(t), \alpha_1(t),
\ldots,
\alpha_n(t))$,
where $\alpha_0(t) \equiv0$ and $\alpha_n(t) \equiv1$,
and $0 \leq\alpha_k(t) \leq1$ for all $k$ and $t$.
We first give a motivating example, which is a special case of the
Shepp--Olkin interpolation.

\begin{example} \label{ex:binomialint}
We write $\operatorname{Bin}_{k}(n, p) := {n\choose k} p^k (1-p)^{n-k}$
for the probability mass function of a binomial with parameters $n$ and $p$.
For fixed $n$ and $0 \leq p < q \leq1$, define $p(t) = p(1-t) + q t$,
and write
$f_k(t) = \operatorname{Bin}_{k}(n, p(t))$
for the probability mass functions which interpolate in the natural way
in the parameter space.
A~simple calculation
(see, e.g., \cite{mateev} and \cite{shepp})
shows that for any $k= 0, 1, \ldots, n$,
%
\begin{equation}
\frac{ \partial f_k}{\partial t}(t) 
=  -
\nabla_1 \bigl( n(q-p) \operatorname{Bin}_{k}\bigl(n-1,
p(t)\bigr) \bigr). \label
{eq:gradient}
\end{equation}
We reformulate equation (\ref{eq:gradient}) using an insight of Yu
\cite{yu},
who defined the hypergeometric thinning operation, observing in \cite{yu},
Lemma~2, that for any $n$, $p$,
%
\begin{equation}
\operatorname{Bin}_{k}(n-1, p) = \frac{ (k+1)}{n} \operatorname
{Bin}_{k+1}(n, p) + \biggl( 1 - \frac
{k}{n} \biggr)
\operatorname{Bin}_{k}(n, p).
\end{equation}
This suggests that we rewrite equation (\ref{eq:gradient}) in the form,
modeled on (\ref{eq:heateqn}),
%
\begin{eqnarray}
0 =  \frac{ \partial f_k}{\partial t}(t) + \nabla_1 \bigl( \bigl(n(q-p) \bigr)
g^{({\bolds{\alpha}})}_k(t) \bigr)\qquad \mbox{for $k= 0, 1, \ldots, n$,}
\label{eq:grad1}
\end{eqnarray}
for
\begin{equation}\qquad g^{({\bolds{\alpha}})}_k(t) =  \alpha_{k+1}(t)
f_{k+1}(t) + \bigl(1-\alpha_k(t)\bigr) f_k(t)\qquad
\mbox{for $k= 0, 1, \ldots , n-1$,}\label{eq:grad2}
\end{equation}
with $\alpha_k(t) = k/n$ for all $k$ and $t$.
\end{example}

The form of equations (\ref{eq:grad1}) and (\ref{eq:grad2}) suggests a
version of the
Benamou--Brenier formula \cite{benamou2,benamou} for discrete random variables.

\begin{definition} \label{def:ourbb}
We write $\PP_{\Z}( f(0), f(1))$ for the set of continuous, piecewise
differentiable families of probability mass functions $f_k$,
given end constraints
$f_k(t) |_{t=0} = f_k(0)$ and $f_k(t) |_{t=1} = f_k(1)$.
Given ${\bolds{\alpha}}(t) \in\AAA$, for
$f_k(t) \in\PP_{\Z}( f(0), f(1))$ define probability mass function
$g^{({\bolds{\alpha}})}_k(t)$, velocity field $v_{\alpha,k}(t)$
and path length $\II
(f)$ by
%
\begin{eqnarray}\qquad
g^{({\bolds{\alpha}})}_k(t) & = &\alpha_{k+1}(t)
f_{k+1}(t) + \bigl(1-\alpha_k(t)\bigr) f_k(t)\qquad
\mbox{for $k=0,1, \ldots, n-1$,} \label{eq:veldef}
\\
0 & =& \frac{ \partial f_k}{\partial t}(t) + \nabla_1 \bigl( v_{\alpha
,k}(t)
g^{({\bolds{\alpha}})}_k(t) \bigr)
\nonumber
\\[-8pt]\label{eq:veldef2}
\\[-8pt]
\eqntext{\mbox{$t$-almost
everywhere for $k= 0, 1, \ldots, n$,}}
\\
\II(f)^2 & =& \int_0^1 \Biggl(
\sum_{k=0}^{n-1} g^{({\bolds
{\alpha}})}_k(t)
v_{\alpha
,k}(t)^2 \Biggr) \,dt =: \int_0^1
\beta(t) \,dt.
\end{eqnarray}
Define $V_n$ via
%
\begin{equation}
\label{eq:vndef} V_n^2\bigl( f(0), f(1) \bigr) = \mathop{
\inf_ {f_k \in\PP_{\Z}( f(0),
f(1)),}}_{ \alpha_k(t) \in\AAA} \II(f)^2 ,
\end{equation}
and refer to any path achieving the infimum in (\ref{eq:vndef}), if it
exists, as a geodesic.
\end{definition}

\begin{proposition} \label{prop:metric}
$V_n$ is a metric on the space of probability measures on $\{0,\ldots,
n\}$, Moreover, for any geodesic
$f$ we have
%
\begin{equation}
\label{eq:lengthspace} V_n \bigl(f(s), f(t)\bigr) = |t-s| V_n
\bigl(f(0), f(1)\bigr) \qquad\mbox{for any $0 \leq s,t \leq1$.}
\end{equation}
\end{proposition}
\begin{pf}
It is clear that $V_n \geq0$ and that $V_n(f,g)=0$ implies $f=g$. To
prove $V_n$ is symmetric, we transpose the path $f(t)$ in time
$\widetilde{f}_k(t) = f_{k}(1-t)$, taking $\widetilde{\alpha}_k(t) =
\alpha_{k}(1-t)$ gives $\widetilde{g}^{(\widetilde{{\bolds
{\alpha}}})}_k(t)
= g^{({\bolds{\alpha}})}_{n-k-1}(1-t)$, and $v_{\widetilde
{\alpha},k} = -v_{\alpha,k}$
so that $V_n^2(f(0), f(1)) = V_n^2( f(1), f(0))$.

To prove the triangle inequality, we consider three mass functions
$f(0),f^*,\break   f(1)$. For any paths $(f^{(0)}(t))_{t \in[0,1]} \in\PP_{\Z
}(f(0),f^*)$ and $(f^{(1)}(t))_{t \in[0,1]} \in\PP_{\Z}(f^*,\break f(1))$ we
construct $(f(t))_{t \in[0,1]} \in\PP_{\Z}(f_0,f_1)$ such that
$ \II(f) = \II(f^{(0)}) + \II(f^{(1)})$, as follows:
\begin{itemize}
\item If $t \leq\rho$, we set $\tau_0(t) := t/\rho$ and $f_k(t) =
f^{(0)}_k(\tau_0(t))$. We then have $\alpha_k(t) = \alpha
^{(0)}_k
(\tau_0(t) )$, with $g^{({\bolds{\alpha}})}_k(t) =
g^{(0),(\alpha)}_k
(\tau
_0(t) )$ and $ v_{\alpha,k}(t) = \frac{1}{\rho} v_{\alpha
,k}^{(0)} (\tau_0(t) )$.
\item If $t > \rho$, we set $\tau_1(t) := (t-\rho)/(1-\rho)$ and
$f_k(t) = f^{(1)}_k(\tau_1(t))$. We have $\alpha_k(t) = \alpha
^{(1)}_k (\tau_1(t) )$, with $g^{({\bolds{\alpha
}})}_k(t) =
g^{(1),(\alpha
)}_k (\tau_1(t) )$ and $ v_{\alpha,k}(t) = \frac
{1}{1-\rho}\times\break 
v_{\alpha,k}^{(1)} (\tau_1(t) )$.
\end{itemize}
The change of variables formula allows us to compute
\begin{eqnarray*}
\II(f)^2 &=& \int_0^\rho\sum
_{k=0}^{n-1} g^{(0),(\alpha)}_k
\bigl(\tau _0(t)\bigr) \frac{1}{\rho^2}v^{(0)}_{\alpha,k}
\bigl(\tau_0(t)\bigr)^2 \,dt
\\
& & {}+ \int_\rho^1 \sum
_{k=0}^{n-1} g^{(1),(\alpha)}_k\bigl(
\tau_1(t)\bigr) \frac{1}{(1-\rho)^2}v^{(1)}_{\alpha,k}
\bigl(\tau_1(t)\bigr)^2 \,dt
\\
&=& \frac{1}{\rho} \II(f_0)^2+ \frac{1}{1-\rho}
\II(f_1)^2.
\end{eqnarray*}
Choosing the optimal $\rho= \rho^* := \frac{\II(f^{(0)})}{\II
(f^{(0)})+\II(f^{(1)})}$ gives the result.

We can extend the same argument to prove equation \eqref
{eq:lengthspace} above. We first prove the case $t=1$.
Consider any geodesic $f$ and $0 \leq s \leq1$. We can take $f^* =
f(s)$ in the argument above
and decompose the geodesic into a path from $f(0)$ to $f(s)$ of length
$\II(f^{(0)})$ and a path from $f(s)$ to $f(1)$ of length $\II
(f^{(1)})$. We know that the optimal $\rho^*$ is
equal to $s$ (otherwise we could reduce $V_n$ by taking the path at a
different rate, contradicting the fact that $f$ is a geodesic).
We deduce that $\II(f^{(0}) = s  ( \II(f^{(0)}) + \II(f^{(1)})
 ) = s V_n(f)$, or that the inner sum $\beta(t) = \sum_{k=0}^{n-1}
g^{({\bolds{\alpha}})}_k(t) v_{\alpha,k}(t)^2$ is
constant almost everywhere in $t$.

We can prove the more general form of equation \eqref{eq:lengthspace}
using a similar argument. We decompose the path
into three parts, $f^{(0)}\in\mathcal{P}_\mathbb{Z}(f(0),f(s))$,
$f^{(1)} \in\mathcal{P}_\mathbb{Z}(f(s),f(t))$ and $f^{(2)} \in
\mathcal{P}_\mathbb{Z}(f(t),f(1))$. Let us consider some $0 < \rho_0 <
\rho_1 < 1$, and take $\tau_0(t) := t/\rho_0$, $\tau_1(t) :=
(t-\rho
_0)/(\rho_1-\rho_0)$ and $\tau_2(t) = (t-\rho_1)/(1-\rho_1)$. A similar
argument shows that unless
$\rho_0 = s$ and
$\rho_1 = t$, the length of the path can be reduced in the same way.
\end{pf}
%
\subsection{Constant velocity paths}

\begin{lemma} \label{lem:lowerbd}
For any geodesic $f(t)$ between
$f(0)$ and $f(1)$, the $\beta(t)$ is constant in $t$. Further, if there
exists a geodesic between these, then writing
mean $\lambda(t) = \sum_k k f_k(t)$
the
\[
V_n\bigl(f(0), f(1) \bigr) \geq\bigl| \lambda(0) - \lambda(1)\bigr |,
\]
with equality if and only if $v_{\alpha,k} \equiv v$ for all $k$ and
$t$, for some $v$.
\end{lemma}

\begin{pf}
Proposition~\ref{prop:metric} shows that for any geodesic, we know that
$\sqrt{\beta(t)} \equiv V_n(f(0),f(1))$ for almost all $t$.
Since $\lambda(t) = \sum_{k=0}^n k f_k(t)$, differentiating and using
equation (\ref{eq:veldef2}) gives
%
\begin{eqnarray}
\label{eq:meanchange} \frac{\partial\lambda}{\partial t}(t) & = & - \sum_{k=0}^n
k \nabla_1 \bigl( v_{\alpha,k}(t) g^{({\bolds{\alpha}})}_k(t)
\bigr)
\nonumber
\\[-8pt]
\\[-8pt]
\nonumber
& = & - \sum_{k=0}^n
\nabla_1^*( k ) \bigl( v_{\alpha,k}(t) g^{({\bolds{\alpha}})}
_k(t) \bigr) = \sum_{l=0}^n
g^{({\bolds{\alpha}})}_l(t) v_{\alpha,l}(t),
\end{eqnarray}
since $-\nabla_1^*(k) = 1$.
Using equation (\ref{eq:veldef2}) and Cauchy--Schwarz, since
$g^{({\bolds{\alpha}})}
(t)$ is a probability mass
function, equation (\ref{eq:meanchange}) gives that for any $t$, since
%
\begin{equation}
\Biggl( \sum_{k=0}^n g^{({\bolds{\alpha}})}_k(t)
v_{\alpha
,k}(t)^2 \Biggr) \geq \Biggl( \sum
_{k=0}^n g^{({\bolds{\alpha}})}_k(t)
v_{\alpha,k}(t) \Biggr)^2 \label{eq:cs} = \biggl(
\frac{ \partial\lambda(t)}{\partial t} \biggr)^2,
\end{equation}
or that $ \llvert  \frac{ \partial\lambda(t)}{\partial t} \rrvert
\leq
\sqrt{ \beta(t)} = V_n(f(0),f(1))$, and the result follows by integration.

Observe that equality holds in equation (\ref{eq:cs}) if and only if
$v_{\alpha,k}(t)
= \frac{\partial\lambda}{\partial t}(t)$ for all~$k$, and equality
holds overall if and only if $\frac{\partial\lambda}{\partial t}(t) =
V_n(f(0), f(1))$ for all $t$.
\end{pf}
Lemma~\ref{lem:lowerbd} focuses attention on interpolations for which
velocity field $v_{\alpha,k}(t) \equiv v$ for all $k$ and $t$, for some
${\bolds{\alpha}(t)} \in\AAA$. Recall that Definition~\ref
{def:modtransp}
refers to such an interpolation as a ``constant velocity path,''
and we say that $f(t)$ satisfies a modified transport equation.

\begin{lemma} \label{lem:unique}
If $f_k(t)$ can be expressed as a constant velocity path for some
choice of $v$ and ${\bolds{\alpha}} \in\AAA$, then this
representation is
unique (there is no other
choice of $v$ and ${\bolds{\alpha}} \in\AAA$ for which it is a constant
velocity path).
\end{lemma}

\begin{pf}
Equation (\ref{eq:meanchange}) shows that if there exists a constant
velocity path with
velocity $v$, then $v = \lambda(1) - \lambda(0)$.
Using a similar argument, we can solve for ${\bolds{\alpha}}$.
The key is
to observe that, since $\alpha_n=1$, for any ${\bolds{\alpha}}
\in\AAA$,
equation (\ref{eq:veldef}) means that the sum
%
\begin{equation}
\label{eq:gandalpha} \sum_{l=k}^{n-1}
g^{({\bolds{\alpha}})}_l(t) = \sum_{l=k}^n
f_l(t) - \alpha_k(t) f_k(t).
\end{equation}
Using the distribution function $F_l(t) := \sum_{k=0}^l f_k(t)$ and
taking $v_{\alpha,l}(t) \equiv v$, we can sum equation (\ref
{eq:veldef2}) over $k$ to obtain
%
\begin{equation}
\label{eq:valdef3} \frac{\partial F_l}{\partial t}(t) + v g^{({\bolds{\alpha
}})}_l(t) = 0.
\end{equation}
Hence, $g^{({\bolds{\alpha}})}_l(t)$ is also fixed by the form
of the path $f_k(t)$, and
on rearranging equation (\ref{eq:gandalpha}), we express
%
\begin{equation}
\label{eq:alphavals} \alpha_k(t) = \frac{ \sum_{l=k}^n f_l(t) - \sum_{l=k}^{n-1}
g^{({\bolds{\alpha}})}
_l(t)}{f_k(t)}.
\end{equation}
\upqed\end{pf}

Equation \eqref{eq:alphavals} implies that $\alpha_k(t)$ is a smooth
function of $t$ in the case of constant velocity paths. In particular,
it is legitimate to consider the derivative $\frac{\partial}{\partial
t} \alpha_k(t)$, as is done, for instance, in the proof of Lemma~\ref
{lem:pde2}.
We now show that in certain circumstances our distance measure $V_n$
coincides with the Wasserstein distance $W_1$, a metric which is known
to have a natural
relationship to discrete interpolations as described in Section~\ref{sec:examples} below.

\begin{lemma} \label{lem:wasserstein}
If there exists a constant velocity path between $f(0)$ and $f(1)$
with velocity $v$,
then the Wasserstein distance $W_1(F(0),F(1))$ and $V_n(f(0), f(1))$ coincide
and are equal to $\lambda(1) - \lambda(0)$.
\end{lemma}

\begin{pf} Recall from the proof of Lemma~\ref{lem:unique} that
if there exists a constant velocity path with
velocity $v$, then $v = \lambda(1) - \lambda(0)$. Without loss of
generality we may assume that $v \geq0$.
In this case, $V_n = v = \lambda(1) - \lambda(0)$.

Using equation (\ref{eq:valdef3}), positivity of $v$
means that $F_k(t)$ is decreasing in $t$ for all $t$ and $k$. This
means that
$F(0)$ stochastically dominates $F(1)$ in the standard sense; see, for
example, \cite{shaked}. Lemma~8.2 and equation (8.1) of \cite{bickel}
together show that for any distribution functions
$F(0)$ and $F(1)$, the $W_1(F(0), F(1)) = \int| F(1)_y - F(0)_y | \,dy$,
so in this stochastically ordered case we deduce that
\begin{eqnarray*}
W_1\bigl(F(0), F(1)\bigr) &=& \int\bigl(F(0)_y -
F(1)_y\bigr) \,dy = \int y \,dF(1)_y - \int y
\,dF(0)_y\\
& =& \lambda(1) - \lambda(0),
\end{eqnarray*}
and the argument is complete.
\end{pf}

\subsection{Binomial interpolation} \label{sec:examples}

\begin{example} \label{ex:bino}
Comparing equations (\ref{eq:grad1}) and (\ref{eq:veldef2})
shows that, taking $\alpha_k(t) \equiv k/n$,
the binomial interpolation
(Example~\ref{ex:binomialint}) has constant velocity $v_{\alpha,k}(t)
\equiv n(q-p)$ and hence achieves the lower bound in Lemma~\ref
{lem:lowerbd}, with $V_n( \operatorname{Bin}(n, p),
\operatorname{Bin}(n, q)) = n(q-p)$.

Contrast this with the approach of Erbar and Maas \cite{erbar,maas}
(see also Mielke~\cite{mielke}), based on
Markov chains with a given stationary distribution ${\bolds{\pi
}}$. In the
two-point case, taking as a reference
${\bolds{\pi}} = (q/(p+q), p/(p+q))$, Maas~\cite{maas} write
$\rho^{\beta}$ for a relative density equivalent to the probability
mass function
$f^{\beta} = (f_0^{\beta},f_1^{\beta}) = ((1-\beta)/2,(1+\beta)/2)$.
Example~2.6 of \cite{maas}
implies a distance of
$ {\mathcal W}( \rho^\alpha, \rho^\beta) = \frac{1}{\sqrt{2 p}}
\int_\alpha^\beta\sqrt\frac{ \arctanh r}{r} \,dr$,
in this case, in contrast to the $|\beta-\alpha|/2$ we obtain.
\end{example}


Example~\ref{ex:bino} can be generalized considerably as follows.
Given a
probability mass function $f$, R\'{e}nyi \cite{renyi4} introduced the
thinned probability mass function $T_t f$ to be the law of the random sum
%
\begin{equation}
\label{eq:thin} \sum_{i=1}^X
B_i^{(t)},
\end{equation}
where $X \sim f$ and $B_i^{(t)}$ are Bernoulli($t$) random variables,
independent of each other and of $X$.
Thinning interpolates between the original measure $f = T_1 f$ and a~point mass at zero
$T_0 f$. This operation was studied in the context of entropy of random
variables in \cite{johnson21},
and was extended by Gozlan et al. \cite{gozlan} and by Hillion~\cite{hillion3}
for probability measures on graphs, implying the following definition
in the case of random variables
supported on $\Z$:

\begin{definition} \label{def:binomint}
A coupling $\pi$ of mass functions $f(0)$ and $f(1)$ supported on $\Z$
[i.e., a joint distribution function $\pi_{x,y}$
whose marginals satisfy $f(0)_x = \sum_y \pi_{x,y}$ and $f(1)_y =
\sum_x \pi_{x,y}$], induces a
path as follows. Section~2.2
of \cite{gozlan} defines a mass function
%
\begin{equation}
\label{eq:gozlan} f_k(t) = v^{\pi}_k(t) := \sum
_{x,y} \pi_{x,y} \operatorname
{Bin}_{k-x}\bigl(|y-x|, t\bigr),
\end{equation}
which we can understand as the law of the random sum
$ 
X + \sum_{i=1}^{Y-X} B_i^{(t)}$, 
where $(X,Y) \sim\pi$, and as before
$B_i^{(t)}$ are Bernoulli($t$) random variables, independent of each
other and of $(X,Y)$.
Here, we use the convention that for $m \geq0$, $\sum_{i=1}^{-m}
B_i^{(t)} = - \sum_{i=1}^m B_i^{(t)}$.
\end{definition}

Proposition~2.7 of \cite{gozlan} gives a partial differential equation
showing how $f_k(t)$
evolves with $t$, using a mixture of left and right gradients (as in
\cite{johnson29}).
Proposition~2.5 of \cite{gozlan} shows that if $\pi^*$ is an optimal coupling
(in Wasserstein distance $W_1$), then $v^{\pi^*}(t)$ defines a
(constant velocity)
geodesic for the $W_1$ distance. We relate this to the discrete
Benamou--Brenier framework in the stochastically ordered context; see
Lemma~\ref{lem:wasserstein}.

\begin{lemma} \label{lem:interpol} If $f(0)$ is stochastically
dominated by $f(1)$,
then the interpolation $(f_k(t))$ defined by equation (\ref{eq:gozlan})
gives a
constant velocity path.
\end{lemma}

\begin{pf}
In this case, $x \leq y$ for all $(x,y)$ in the support of $\pi$.
Define $v = \sum_{x,y} \pi_{x,y}(y-x)$ and $\wt{\pi}_{x,y} = \pi
_{x,y}(y-x)/v$ for another
``distance-biased'' joint distribution function. Direct
differentiation of equation (\ref{eq:gozlan}) gives
that
\begin{eqnarray*}
\frac{\partial f_k}{\partial t}(t)& = &\sum_{x,y}
\pi_{x,y}(y-x) \bigl( \operatorname{Bin}_{k-x-1}(y-x-1, t) -
\operatorname {Bin}_{k-x}(y-x-1, t) \bigr)\\
&=& - \nabla_1\bigl(
v g_k(t)\bigr),
\end{eqnarray*}
where $g_k(t) = \sum_{x,y} \wt{\pi}_{x,y} \operatorname
{Bin}_{k-x}(y-x-1, t)$.
Since for any $x$, the convolution
$(1-t) \operatorname{Bin}_{x}(m-1, t) + t \operatorname
{Bin}_{x-1}(m-1, t) = \operatorname{Bin}_{x}(m, t)$,
in equation (\ref{eq:gozlan}) we can express
%
\begin{equation}
\label{eq:split} \quad f_l(t) = \sum_{x,y}
\pi_{x,y} \bigl( (1-t) \operatorname {Bin}_{l-x}(y-x-1, t) + t
\operatorname{Bin}_{l-x-1}(y-x-1, t) \bigr),
\end{equation}
and substituting in equation (\ref{eq:alphavals}), we obtain
%
\begin{eqnarray}
\label{eq:alphaval} \alpha_k(t) & = &  \sum_{x,y} \pi_{x,y}  \Biggl[ t
\operatorname{Bin}_{k-x-1}(y-x, t)\nonumber
\nonumber
\\[-8pt]
\\[-8pt]
\nonumber
&&{}+ \sum_{l=k}^{n-1} \operatorname{Bin}_{l-x}(y-x-1, t)  \bigl( 1-
{(y-x)}/{v} \bigr)
 \Biggr]\bigg/{f_k(t)}.\nonumber
\end{eqnarray}
\upqed\end{pf}
In future work, we hope to consider
the question of which interpolations in the form of equation (\ref
{eq:gozlan}) induce coefficients satisfying
$0 \leq\alpha_k(t) \leq1$.

\begin{example}[(Translation case)] \label{ex:translation}
Hillion considered the translation case, where $f(1)_{k+m} = f(0)_k =
f_k$ for some $m$.
Theorem~1.1 of \cite{hillion2} proved that if $f$ is
log-concave (i.e., $f_k^2 \geq f_{k-1} f_{k+1}$ for all $k$), the
entropy is concave in $t$.
This paper generalizes Hillion's
result: the conditions of Theorem~\ref{th:ConvexityEntropy} can be verified, and the concavity of entropy
is reproved.

In particular
we interpolate by $\pi_{x,y}$ supported only on $\{ (x,y)\dvtx y-x = m \}$,
so that $\wt{\pi} = \pi$, and
clearly $v = m$. Then equation (\ref{eq:split}) simplifies to give
$f_k(t) = (1-t) g_{k}(t) + t g_{k-1}(t)$,
and equation (\ref{eq:alphaval}) becomes
%
\begin{equation}
\alpha_k(t)  = \frac{t g_{k-1}(t)}{ f_k(t)} = 1 -
\frac{(1-t)
g_k(t)}{f_k(t)}, \label{eq:transalphachoice}
\end{equation}
so that clearly $\alpha_k(t)$ lies between $0$ and $1$ for all $k$ and $t$.\vadjust{\goodbreak}

Equation (\ref{eq:transalphachoice}) shows that GLC, Condition
\ref{cond:glc}, holds with equality in this case. Further
$k$-monotonicity, Condition \ref{cond:kmon}, holds as a consequence of
log-concavity
of $g_k(t)$ (which follows from log-concavity of $f$). The
$t$-monotonicity, Condition \ref{cond:tmon}, is
less straightforward, but can be verified using direct calculation,
using the log-concavity of $h$.
\end{example}

\subsection{Generalized log-concavity} \label{sec:glc}
Recall from Section~\ref{sec:benbren} that probabilistic inequalities
can be proved for densities $f$ such that $f/\phi_{1/c}$ is log-concave.
For integer-valued random variables the corresponding property of ultra
log-concavity was introduced by Pemantle \cite{pemantle} and promoted by
Liggett \cite{liggett}:

\begin{definition}[(\cite{pemantle,liggett})] \label{def:ulc}
For any $n$, a probability mass function supported on $\{ 0, 1, \ldots,
n \}$ is ultra log-concave of
order $n$, denoted by ${\operatorname{\mathbf{ULC}}}(n)$, if the ratio
$ f_k/\operatorname{Bin}_{k}(n, t)$ is a log-concave function.
Equivalently we
require that
%
\begin{eqnarray}
\label{eq:ulcn} \frac{k+1}{n} \biggl( 1- \frac{k+1}{n} \biggr)
f_{k+1}^2 - \frac{k+2}{n} \biggl( 1- \frac{k}{n}
\biggr) f_{k} f_{k+2} \geq0\qquad
\mbox {for $k=0, \ldots, n-2$}.\nonumber\\
\end{eqnarray}
We include the possibility that (formally speaking) $n=\infty$, in
which case we require that the ratio of
$f_k$ and a Poisson mass function be log-concave, and write ${\operatorname{\mathbf
{ULC}}}(\infty)$.
\end{definition}

This condition was first used to control entropy by
Johnson \cite{johnson21}, who showed that, fixing the mean, the Poisson
is maximum entropy in the class ${\operatorname{\mathbf{ULC}}}(\infty
)$. A~corresponding result
was proved by Yu \cite{yu2}, who showed that, fixing the mean, the
binomial is maximum
entropy in the class ${\operatorname{\mathbf{ULC}}}(n)$. This generalizes
the result (see~\cite{mateev} and~\cite{shepp}) that the entropy of
Bernoulli sums with a given mean
is maximized by the binomial, since Newton's inequalities (see, e.g.,
\cite{niculescu}) show that Bernoulli sums are ${\operatorname{\mathbf
{ULC}}}(n)$.

Our generalized log-concavity, Condition \ref{cond:glc}, generalizes
Definition~\ref{def:ulc}, with ${\operatorname{\mathbf{ULC}}}(n)$
corresponding to ${\mathbf
{GLC}}({\bolds{\alpha}})$
for $\alpha_k = k/n$, as in Example~\ref{ex:binomialint}.
Note that GLC, Condition \ref{cond:glc}, and $k$-monotonicity,
Condition \ref{cond:kmon},
together imply that $f$ is log-concave.

\section{Framework for concavity of discrete entropy} \label{sec:entcon}

In this section,
we prove
Theorem~\ref{th:ConvexityEntropy}, which shows that entropy is concave
if Conditions \ref{cond:kmon}, \ref{cond:tmon} and \ref{cond:glc} are
satisfied.
In fact, since $t$-monotonicity (Condition \ref{cond:tmon}) is too
restrictive for our purposes, we
prove a more general result, Theorem~\ref{thm:cond1-4}, which gives
concavity of entropy despite replacing Condition \ref{cond:tmon} by the
weaker Condition \ref{cond:delta}.
Lemma~\ref{lem:fghIneq1} shows that this condition is indeed weaker,
and hence together with Theorem~\ref{thm:cond1-4} proves
Theorem~\ref{th:ConvexityEntropy}.

\begin{condition} \label{cond:delta}
Consider a constant velocity path, satisfying a modified transport equation
$\frac{\partial f_k}{\partial t}(t) = - v \nabla_1  (
g^{({\bolds{\alpha}})}_k(t)
 )$
with some $h$ satisfying $\frac{\partial^2 f_k}{\partial t^2}(t) = v^2
\nabla_2  ( h_k  )$.
If we define
%
\begin{equation}
\label{eq:htdef} \wt{h}_k := \frac{ 2 g^{({\bolds{\alpha}})}_k g^{({\bolds
{\alpha}})}_{k+1} f_{k+1} -  (
g^{({\bolds{\alpha}})}_k
 )^2 f_{k+2} -  ( g^{({\bolds{\alpha}})}_{k+1}
)^2 f_{k}}{
f_{k+1}^2 -
f_k f_{k+2}},
\end{equation}
then we require that
%
\begin{equation}
\label{eq:nec} h_k \leq\wt{h}_{k}\qquad \mbox{for $k = 0, 1,
\ldots, n-2$.}
\end{equation}
\end{condition}

We first observe that
the same coefficients $(\alpha_k)_{k =0, \ldots, n}$ introduced in
equation~(\ref{eq:mixture})
can be used to state a second-order modified transport equation:

\begin{lemma} \label{lem:pde2}
If there exist coefficients ${\bolds{\alpha}}$ giving rise to an
interpolation with constant velocity $v$, then
%
\begin{equation}
\frac{\partial^2 f_k}{\partial t^2} = v^2 \nabla_2 ( h_k ),
\end{equation}
where [in a result paralleling (\ref{eq:hexpress}) in the continuous
case above], for $k=0, \ldots, n-2$,
%
\begin{eqnarray}\label{eq:ModifiedTransport2}
h_k  &=&  (1-\alpha_k) (1-\alpha_{k+1})
f_k + 2\alpha_{k+1}(1-\alpha _{k+1})
f_{k+1} + \alpha_{k+1} \alpha_{k+2} f_{k+2}
\nonumber
\\[-8pt]
\\[-8pt]
\nonumber
&&{}-
f_{k+1}\frac{1}{v} \frac{\partial\alpha_{k+1}}{\partial t}.
\end{eqnarray}
\end{lemma}

\begin{pf} Recall that we write
%
\begin{equation}
g^{({\bolds{\alpha}})}_k = \alpha_{k+1}(t) f_{k+1}(t) +
\bigl(1-\alpha_k(t)\bigr) f_k(t) \qquad\mbox{for $k=0,1,
\ldots, n-1$}. \label{eq:newveldef}
\end{equation}
Differentiating equation \eqref{eq:newveldef} we have
\begin{eqnarray*}
\frac{\partial g^{({\bolds{\alpha}})}_k}{\partial t}
&=&f_{k+1}\frac{\partial\alpha_{k+1}}{\partial t}+ \alpha _{k+1}\frac
{\partial f_{k+1}}{\partial t}
+ (1-\alpha_k)\frac{\partial
f_k}{\partial t}-f_k\frac{\partial\alpha_k}{\partial t}
\\
&=& \biggl[f_{k+1}\frac{\partial\alpha_{k+1}}{\partial t} -v (1-\alpha _{k+1})
g^{({\bolds{\alpha}})}_k - v\alpha_{k+1} g^{({\bolds{\alpha}})}_{k+1}
\biggr] \\
&&{}- \biggl[f_k\frac
{\partial\alpha_k}{\partial t}-v(1-\alpha_k)
g^{({\bolds
{\alpha}})}_{k-1}- v\alpha _{k} g^{({\bolds{\alpha}})}_k
\biggr]
\\
&=& \nabla_1 [v h_k ],
\end{eqnarray*}
using the expression from (\ref{eq:ModifiedTransport2}),
and the proposition follows easily.\vadjust{\goodbreak}
\end{pf}

\begin{lemma} \label{lem:fghIneq1}
If Conditions \ref{cond:kmon}, \ref{cond:tmon} and
\ref{cond:glc} hold, then Condition \ref{cond:delta} holds.
\end{lemma}

\begin{pf}
Using
$\wt{h}_k$ and $h_k$ defined in equations (\ref{eq:htdef}) and
(\ref{eq:ModifiedTransport2}), we need to prove that $h_k \leq\wt
{h}_k$ for all $k$.
For simplicity, we write
%
\begin{eqnarray}
D_k &:=& f_{k}^2 - f_{k-1}
f_{k+1}  \geq 0,
\\
A_k& :=& \bigl(f_{k+1} g^{({\bolds{\alpha}})}_k -
f_k g^{({\bolds
{\alpha}})}_{k+1}\bigr)  \geq 0, \label
{eq:akdef}
\\
B_k &:=& \bigl(f_{k+1} g^{({\bolds{\alpha}})}_{k+1} -
f_{k+2} g^{({\bolds{\alpha}})}_k\bigr)  \geq 0. \label
{eq:bkdef}
\end{eqnarray}
The positivity of $A_k$ and $B_k$ follows from GLC and $k$-monotonicity since
we can write
\begin{eqnarray*}
(1- \alpha_{k+1}) A_k & = & \operatorname{GLC}({\bolds{\alpha
}})_k + (\alpha _{k+1} - \alpha_k)
f_k g^{({\bolds{\alpha}})}_{k+1},
\\
\alpha_{k+1} B_k & = & \operatorname{GLC}({\bolds{\alpha}})_k
+ (\alpha_{k+2} - \alpha_{k+1}) f_{k+2}
g^{({\bolds{\alpha}})}_{k}.
\end{eqnarray*}
The key is to observe that in this notation, by Lemma~\ref{lem:pde2},
%
\begin{eqnarray}
\wt{h}_k & = & \frac{g^{({\bolds{\alpha}})}_{k+1} A_k +
g^{({\bolds{\alpha}})}_k B_k}{D_{k+1}}, \label
{eq:htildenew}
\\
h_k & = & \alpha_{k+1} g^{({\bolds{\alpha}})}_{k+1} +
(1-\alpha _{k+1}) g^{({\bolds{\alpha}})}_{k} - f_{k+1}
\frac{1}{v} \frac{\partial\alpha_{k+1}}{\partial t}. \label
{eq:hknew}
\end{eqnarray}
Direct calculation gives
%
\begin{eqnarray}
g^{({\bolds{\alpha}})}_k D_{k+1} & = & f_{k+1}
A_k + f_k B_k, \label{eq:dexpr1}
\\
g^{({\bolds{\alpha}})}_{k+1} D_{k+1} & = & f_{k+2}
A_k + f_{k+1} B_k. \label{eq:dexpr2}
\end{eqnarray}
Considering the coefficients of $A_k$ and $B_k$,
we can substitute (\ref{eq:dexpr1}) and (\ref{eq:dexpr2}) in equation
(\ref{eq:hknew}) to obtain
%
\begin{eqnarray}
\label{eq:ktmon} \wt{h}_k - h_k = \frac{ f_{k+2} ( \alpha_{k+2} - \alpha_{k+1}) A_k + f_{k} (\alpha
_{k+1}-\alpha_{k}) B_k}{D_{k+1}} +
f_{k+1}\frac{1}{v} \frac{\partial\alpha_{k+1}}{\partial t}
\geq 0,\nonumber\\
\end{eqnarray}
where the
positivity of $\frac{\partial\alpha_{k+1}}{\partial t} $ is assumed in
Condition \ref{cond:tmon}.
\end{pf}

We need one further result, which can be directly verified by differentiation:

\begin{lemma} \label{lem:logs}
Writing $\theta(v) =1/(2 v) - v/2$, we have
\[
0 \leq-\log v \leq\theta(v)\qquad \mbox{for $v \leq1$.}
\]
\end{lemma}


\begin{theorem} \label{thm:cond1-4}
If Conditions \ref{cond:kmon}, \ref{cond:glc} and \ref{cond:delta}
hold, then
the entropy of $H(f)$ is concave in $t$.
\end{theorem}

\begin{pf} For simplicity, we write $g_k$ for $g_k^{(\alpha)}$.
First note that Conditions \ref{cond:kmon} and \ref{cond:glc}
together imply
that $f$ is log-concave. In fact, they imply two stronger results, that
%
\begin{equation}
\frac{ f_k g_{k+1}}{f_{k+1} g_k} \leq1 \quad\mbox{and}\quad \frac{ f_{k+2} g_{k}}{f_{k+1} g_{k+1}} \leq1. \label{eq:score2}
\end{equation}
This means that, using Lemma~\ref{lem:logs}, writing $\theta(v) =1/(2
v) - v/2$, we can write
%
\begin{eqnarray}
\label{eq:replace} 0 \leq- \log \biggl( \frac{f_k(t) f_{k+2}(t)}{f_{k+1}(t)^2} \biggr) & = & - \log
\biggl( \frac{ f_k g_{k+1}}{f_{k+1} g_k} \biggr) - \log \biggl( \frac{ f_{k+2} g_{k}}{f_{k+1} g_{k+1}} \biggr)
\nonumber
\\
& \leq& \theta \biggl( \frac{ f_k g_{k+1}}{f_{k+1} g_k} \biggr) + \theta \biggl(
\frac{ f_{k+2} g_{k}}{f_{k+1} g_{k+1}} \biggr)
\\
& = & \frac{D_{k+1}}{2 f_{k+1} g_k g_{k+1}} \biggl( \frac{g_k^2}{f_k} + \frac{g_{k+1}^2}{f_{k+2}}
\biggr),\nonumber
\end{eqnarray}
where the last identity follows by grouping together multiples of
$g_{k+1}/g_k$ and $g_k/g_{k+1}$ and factorizing.
In a standard fashion, we can write the second derivative of entropy as
%
\begin{eqnarray}
H''(t) & = & -\sum_{k =0}^n
\frac{\partial^2 f_k}{\partial t^2}(t) \log \bigl(f_k(t)\bigr) - \sum
_{k =0}^n \frac{1}{f_k} \biggl(
\frac{\partial f_k}{\partial t}(t) \biggr)^2
\nonumber
\\
& = & - \sum_{k =0}^n v^2
\nabla_2 (h_k ) \log\bigl(f_k(t)\bigr) - \sum
_{k =0}^n \frac{ ( \nabla_1( v g_k)  )^2
}{f_k} \label{eq:tocontrol}
\\
& = & v^2 \sum_{k =0}^n
h_k \biggl( - \log \biggl( \frac{f_k(t) f_{k+2}(t)}{f_{k+1}(t)^2} \biggr) \biggr) - \sum
_{k =0}^n \frac{ ( \nabla_1( v g_k)  )^2 }{f_k} \label{eq:tocontrol2}
\\
& \leq& v^2 \sum_{k =0}^n
\biggl[ \wt{h}_k \biggl( \frac{D_{k+1}}{2 f_{k+1} g_k g_{k+1}} \biggl( \frac
{g_k^2}{f_k}
+ \frac{g_{k+1}^2}{f_{k+2}} \biggr) \biggr)- \biggl( \frac{ g_k^2}{f_k} - 2
\frac{g_k g_{k+1}}{f_{k+1}} + \frac
{g_{k+1}^2}{f_{k+2}} \biggr) \biggr], \label{eq:tocontrol3}
\end{eqnarray}
where equation (\ref{eq:tocontrol}) follows
using equations (\ref{eq:hkdef}) and (\ref{eq:modified}), respectively,
and equation~(\ref{eq:tocontrol2}) uses the adjoint of $\nabla_2$.
Finally equation (\ref{eq:tocontrol3}) follows from Condition \ref{cond:delta}
and equation (\ref{eq:replace}), using the fact that both terms are positive.

Using equation (\ref{eq:htildenew}), we can write the
first term in the square bracket in equation~(\ref{eq:tocontrol3}) as
\[
\frac{(g_{k+1} A_k + g_k B_k)}{2 f_{k+1} g_k g_{k+1}} \biggl( \frac{g_k^2}{f_k} + \frac{g_{k+1}^2}{f_{k+2}} \biggr).
\]
Further, since we can write
\[
- \biggl( \frac{ g_k^2}{f_k} - 2 \frac{g_k g_{k+1}}{f_{k+1}} + \frac
{g_{k+1}^2}{f_{k+2}} \biggr)
= - \frac{g_k}{f_k f_{k+1}} A_k - \frac{g_{k+1}}{f_{k+1} f_{k+2}} B_k,
\]
we can expand equation (\ref{eq:tocontrol3}) in terms of $A_k$ and
$B_k$ as
%
\begin{eqnarray}
\label{eq:remainderterm} &&- \frac{ ( A_k g_{k+1} - B_k g_k  )
 ( f_{k+2} g_k^2 - f_k g_{k+1}^2  )}{2 f_k f_{k+1} f_{k+2}
g_k g_{k+1}}
\nonumber
\\[-8pt]
\\[-8pt]
\nonumber
&&\qquad= - \frac{ f_k f_{k+1} f_{k+2}}{2 g_k g_{k+1} } \biggl(
\frac{ g_k^2}{f_k
f_{k+1}} - \frac{ g_{k+1}^2}{f_{k+1} f_{k+2}} \biggr)^2.
\end{eqnarray}
Here the final equality follows
since the form of $A_k$ and $B_k$ means that the two bracketed terms in
the first expression in equation (\ref{eq:remainderterm}) are in fact equal.
\end{pf}
It would be of interest to understand how this remainder term (\ref
{eq:remainderterm}) relates to the corresponding term found for
continuous interpolations in
equation (\ref{eq:velchange}).
%
\section{Shepp--Olkin interpolation as a constant velocity path}
\label
{sec:soint}

Recall that in Conjecture \ref{conj:SObis}
we are given $n\geq1$ affine functions $p_i \dvtx [0,1] \rightarrow[0,1]$
where for each $i$,
$p_i(t) = p_i(0)(1-t) + p_i(1) t$.
We denote by $(f_k(t))_{k =0,1,
\ldots, n}$
the distribution of the sum of independent Bernoulli random variables of
parameters $p_1(t),\ldots, p_n(t)$. Further we write $(f_k^{(i)}(t))_{k
=0, \ldots, n-1}$ for the mass function
of the $i$th ``leave one out'' sum---that is, the distribution of a
Bernoulli sum with parameters
$p_1(t),\ldots,p_{i-1}(t),p_{i+1}(t), \ldots, p_n(t)$, and
$f_k^{(i,j)}(t)$ for the ``leave two out'' sum, involving
all parameters except $p_i(t)$ and $p_j(t)$. Define
%
\begin{eqnarray}
D_k^{(i)} & = & \bigl( f^{(i)}_k\bigr)
^2 - f_{k-1}^{(i)} f_{k+1}^{(i)},\label
{eq:dkdef}
\\
E_k^{(i)} & = & f^{(i)}_k
f^{(i)}_{k-1} - f_{k-2}^{(i)}
f_{k+1}^{(i)}, \label{eq:ekdef}
\end{eqnarray}
with corresponding notation for $D_k$, $D_k^{(i,j)}$ and so on.

We now show how the Shepp--Olkin problem can be viewed in the framework
we introduced in Section~\ref{sec:discrete}. To be
specific,
if each derivative $p_i'$ is positive, the Shepp--Olkin interpolation
is a
constant velocity path with velocity $v$ in the sense of Definition~\ref
{def:modtransp}.
That is:

\begin{proposition} \label{prop:alphaformula} If all the $p_i' \geq
0$, then
the probability mass function defined by the Shepp--Olkin interpolation
satisfies
a modified transport equation,
%
\begin{equation}
\label{eq:ModifiedTransport} \frac{\partial f_k}{\partial t}(t) + \nabla_1( v g_k) =
0.
\end{equation}
Here we set $v:= \sum_{i=1}^n p_i'$,
and write the probability mass function
%
\begin{equation}
\label{eq:gdef} g_k(t) := \bigl(1-\alpha_k(t)\bigr)
f_k(t) + \alpha_{k+1}(t) f_{k+1}(t),
\end{equation}
where
%
\begin{equation}
\alpha_k(t) = \frac{\sum_{i=1}^n p_i' p_i(t) f_{k-1}^{(i)}(t)}{v f_k(t)} \label{eq:alphaDefinition} = 1-
\frac{\sum_{i=1}^n p_i' (1-p_i(t)) f_k^{(i)}(t)}{v f_k(t)}.
\end{equation}
Observe that $\alpha_0(t) \equiv0$ and $\alpha_n(t) \equiv1$, and that
if all $p_i' \geq0$, then $0 \leq\alpha_k(t) \leq1$ for all $k$ and $t$.
Further, this interpolation satisfies a second-order modified transport
equation of the form
$\frac{\partial^2 f_k}{\partial t^2}(t) = v^2 \nabla_2  ( h_k
)$, where
%
\begin{equation}
\label{eq:sohk} h_k = \frac{\sum_{i \neq j} p_i' p_j' f_k^{(i,j)}}{v^2}.
\end{equation}
\end{proposition}

\begin{pf}
Observe that by definition, the $f_k$ have probability generating function
\[
\sum_{k=0}^n f_k(t)
s^k = \prod_{i=1}^n
\bigl(1-p_i(t) + s p_i(t) \bigr),
\]
which has derivative with respect to $t$ given by
%
\begin{equation}
\label{eq:pgfder1} \sum_{k=0}^n
\frac{\partial
f_k}{\partial t}(t) s^k = \sum_{i=1}^n
p_i'(s-1) \prod_{j \neq i}
\bigl(1-p_j(t) + s p_j(t) \bigr).
\end{equation}
Comparing coefficients of $s$, we see that
$\frac{\partial f_k}{\partial t} + \nabla_1  ( v \widetilde{g}_k(t)
 ) =0$,
where $\widetilde{g}_k(t) := \frac{1}{v} \sum_{i=1}^n p_i' f_k^{(i)}$.
Substituting equation (\ref{eq:alphaDefinition}) in equation (\ref
{eq:gdef}), we obtain
that $\widetilde{g}_k(t) = g_k(t)$, and so equation (\ref
{eq:ModifiedTransport}) follows.

The values of $\alpha_0(t)$ and $\alpha_n(t)$ follow from the fact that
$f_{-1}^{(i)} = f_{n}^{(i)} = 0$.
The positivity of $\alpha_k(t)$ and $1-\alpha_k(t)$ follow from
the assumption that $p_i' \geq0$, meaning that all the terms in both
the fractions in equation
(\ref{eq:alphaDefinition}) are positive.

The form of $h_k$ stated in equation (\ref{eq:sohk}) follows by taking
a further derivative of equation (\ref{eq:pgfder1})
to obtain
%
\begin{equation}
\label{eq:pgfder2} \sum_{k=0}^n
\frac{\partial^2
f_k}{\partial t^2}(t) s^k = \sum_{i=1}^n
\sum_{ j \neq i} p_i'
p_j' (s-1)^2 \prod
_{\ell\neq i,j} \bigl(1-p_\ell(t) + s p_\ell(t)
\bigr),
\end{equation}
and again the result follows on comparing coefficients of $s$.
\end{pf}
For clarity, we now suppress the explicit dependence of $\alpha_k$ on $t$.
We first verify the $k$-monotonicity Condition \ref{cond:kmon}
for Shepp--Olkin interpolations with $p_i' \geq0$.%

\begin{proposition} \label{prop:alphaMonotonicity}
For the Shepp--Olkin interpolation described above, if $p_i' \geq0$
for all $i$, then
the coefficients $(\alpha_k)_{k = 0, 1, \ldots, n}$ satisfy the
inequality $\alpha_k \leq\alpha_{k+1}$;
that is, $k$-monotonicity (Condition \ref{cond:kmon}) holds.
\end{proposition}

\begin{pf}
If $k$ is such that $f_k>0$ and $f_{k+1}>0$, then
\[
\alpha_{k+1}-\alpha_{k} = \frac{\sum_{i=1}^n p_i'p_i  [f_k^{(i)}
f_k - f_{k-1}^{(i)} f_{k+1}  ]}{(\sum_{i=1}^n p_i') f_k f_{k+1}}.
\]
Moreover, for $i \in\{1,\ldots, n\}$, we have
\begin{eqnarray*}
&&f_k^{(i)} f_k - f_{k-1}^{(i)}
f_{k+1} \\
&&\qquad= f_{k}^{(i)} \bigl[(1-p_i)f_{k}^{(i)}+p_i
f_{k-1}^{(i)} \bigr]-f_{k-1}^{(i)}
\bigl[(1-p_i)f_{k+1}^{(i)}+p_i
f_{k}^{(i)} \bigr]
\\
&&\qquad= (1-p_i) \bigl[\bigl(f_{k}^{(i)}
\bigr)^2-f_{k-1}^{(i)}f_{k+1}^{(i)}
\bigr] \geq0,
\end{eqnarray*}
by the log-concavity property of the Bernoulli sum $f^{(i)}$. The fact
that each $p_i'$ is nonnegative finally proves that $\alpha_k \leq
\alpha_{k+1}$.
\end{pf}

\begin{remark} \label{rem:symm}
In the Shepp--Olkin case, $\alpha_k$ is a conditional expectation of a weighted
sum, similarly to the ``scaled score function'' of \cite{johnson11}.
This follows since writing $B_i$ for a
Bernoulli random variable with parameter $p_i(t)$, we obtain
$\pr(B_i = 1 | B_1 + \cdots+ B_n = k) = p_i(t)
f_{k-1}^{(i)}(t)/f_k(t)$ so that
equation \eqref{eq:alphaDefinition} can be expressed as
%
\begin{equation}
\alpha_k(t) = \ep \Biggl( \sum_{i=1}^n
\lambda_i B_i \Bigg\vert \sum
_{i=1}^n B_i = k \Biggr),
\end{equation}
where weights $\lambda_i = p_i'/(\sum p_i')$. Note that in particular,
in the ``symmetric'' case where $p_i' \equiv p'$ for all $i$, then
$\lambda_i' = 1/n$ and $\alpha_k(t) \equiv k/n$.

This conditional expectation characterization
allows us to give an alternative proof of the $k$-monotonicity
Proposition~\ref{prop:alphaMonotonicity}. A
result of Efron \cite{efron} (see also~\cite{johnson13}, equation
(3.1)) shows that if
$\phi( u_1, \ldots, u_n)$ is an increasing function in each variable
and $X_1, \ldots, X_n$ are independent log-concave
random variables, then $\Phi(k) := \ep [ \phi(X_1, \ldots,
X_n) |
X_1 + \cdots+ X_n = k  ]$ is
an increasing function of $k$. Applying this to $\phi(B_1, \ldots, B_n)
= \sum\lambda_i B_i$, the result follows.
\end{remark}

We now prove that in the monotone Shepp--Olkin
case the Bernoulli sum mass function is ${\mathbf{GLC}}({\bolds
{\alpha }})$, for the
natural choice of ${\bolds{\alpha}}$, and hence Condition \ref
{cond:glc} holds.

\begin{proposition} \label{prop:GLC} For the Shepp--Olkin interpolation,
taking ${\bolds{\alpha}}$ as defined in equation (\ref
{eq:alphaDefinition}),
if all the $p_i'$ are positive, then
the Bernoulli sum mass function $(f_k)_{k = 0, 1, \ldots, n}$ is
${\mathbf{GLC}}({\bolds{\alpha}})$, and Condition
\ref{cond:glc} holds.
\end{proposition}

\begin{pf}
Formula \eqref{eq:alphaDefinition} shows that $\operatorname{GLC}({\bolds
{\alpha}})_k$ can be
written as $1/v^2$ times
%
\begin{eqnarray}
\label{eq:GLCequiv} &&\biggl(\sum_i
p_i' p_i f_{k}^{(i)}
\biggr)\cdot \biggl(\sum_i p_i'
(1-p_i) f_{k+1}^{(i)} \biggr)
\nonumber
\\[-8pt]
\\[-8pt]
\nonumber
&&\qquad{}- \biggl(\sum
_i p_i'
(1-p_i) f_{k}^{(i)} \biggr)\cdot \biggl(\sum
_i p_i'
p_i f_{k+1}^{(i)} \biggr).
\end{eqnarray}
Expanding this expression as a quadratic form in $p_1', \ldots, p_n'$,
the coefficients of $p_i'^2$ vanish, leaving an expression which
simplifies to
\[
\sum_{i < j} p_i'
p_j' (p_i - p_j)^2
\bigl[ \bigl( f_k^{(i,j)} \bigr)^2 -
f_{k-1}^{(i,j)} f_{k+1}^{(i,j)} \bigr].
\]
The positivity of this expression, and hence the ${\mathbf
{GLC}}({\bolds{\alpha}})$
property, follow from
the log-concavity of $(f_{k}^{(i,j)})_{k = 0, \ldots, n-2}$ and
positivity of $p_i'$.
\end{pf}

\section{Entropy concavity for monotone Shepp--Olkin regime} \label
{sec:soconc}

We now show that entropy is concave in the monotone Shepp--Olkin regime.
Having already verified Conditions \ref{cond:kmon} and \ref{cond:glc}
in Propositions \ref{prop:alphaMonotonicity} and \ref{prop:GLC},
Theorem~\ref{thm:cond1-4} shows that
concavity of entropy follows if
Condition \ref{cond:delta} holds.

\begin{proposition} For monotone Shepp--Olkin interpolations,
Condition \ref{cond:delta} holds if
%
\begin{eqnarray}
\label{eq:diffexp}&& \sum_{i < j} \bigl(
p_i'^2 p_j(1-p_j)
b_{i,j} + p_j'^2
p_i(1-p_i) b_{j,i}
\nonumber
\\[-8pt]
\\[-8pt]
\nonumber
&&\hspace*{13pt}\qquad{}+ 2 p_i'
p_j' p_i (1-p_i)
p_j (1-p_j) c_{i,j} \bigr) \geq0,
\end{eqnarray}
where $b_{i,j}$ and $c_{i,j}$ are defined in equations (\ref{eq:bdef})
and (\ref{eq:cdef}) below.
\end{proposition}

\begin{pf}
We use the fact that [in the notation of equations (\ref{eq:akdef}) and
(\ref{eq:bkdef})] the numerator of $\wt{h}_k$
can be written as $g_{k+1} A_k + g_k B_k$, where
\begin{eqnarray*}
A_k &:=& (f_{k+1} g_k - f_k
g_{k+1})  =  \frac{1}{v} \sum_i
p_i' p_i D_k^{(i)}
\geq0,
\\
B_k &:=& (f_{k+1} g_{k+1} - f_{k+2}
g_k)  =  \frac{1}{v} \sum_i
p_i' (1-p_i) D^{(i)}_{k+1}
\geq0.
\end{eqnarray*}
This means [using the expression for $h_k$ from equation (\ref
{eq:sohk}) above] that Condition~\ref{cond:delta} is equivalent to the
positivity of
%
\begin{eqnarray}
\label{eq:toprove}&& v^2 ( g_{k+1} A_k +
g_k B_k ) - D_{k+1} \sum
_{i \neq j} p_i' p_j'
f_k^{(i,j)}
\nonumber
\\
&&\qquad =  \sum_i \bigl(p_i'
\bigr)^2 \bigl( f_{k+1}^{(i)} p_i
D_k^{(i)} + f_k^{(i)}
(1-p_i) D_{k+1}^{(i)} \bigr)
\\
& &\qquad\quad{} + \sum_{i \neq j} p_i'
p_j' \bigl( f_{k+1}^{(j)}
p_i D_k^{(i)} + f_k^{(j)}
(1-p_i) D_{k+1}^{(i)\nonumber} - D_{k+1}
f^{(i,j)}_k \bigr).
\end{eqnarray}
We expand the two bracketed terms in equation (\ref{eq:toprove}) in
terms of $f_k^{(i,j)}$, using Lemma~\ref{lem:soiBvar} below, which implies [using the notation of equations
(\ref{eq:dkdef})
and (\ref{eq:ekdef})]
that
\begin{eqnarray*}
f_k^{(i)} f_{k-1}^{(i)} & = & \sum
_{j \neq i} p_j (1-p_j)
D_{k-1}^{(i,j)},
\\
f_k^{(i)} f_{k-2}^{(i)} & = &
\frac{1}{2} \sum_{j \neq i} p_j
(1-p_j) E_{k-1}^{(i,j)}.
\end{eqnarray*}
First observe that
\begin{eqnarray*}
&&f_{k+1}^{(i)} p_i D_k^{(i)}
+ f_k^{(i)} (1-p_i) D_{k+1}^{(i)}
\\
&&\qquad =  p_i f_{k}^{(i)} \bigl[
f_k^{(i)} f_{k+1}^{(i)} \bigr] -
p_i f_{k+1}^{(i)} \bigl[ f_{k-1}^{(i)}
f_{k+1}^{(i)} \bigr] + (1-p_i)
f_{k+1}^{(i)} \bigl[ f_k^{(i)}
f_{k+1}^{(i)} \bigr]\\
&&\qquad\quad{} - (1- p_i)
f_{k}^{(i)} \bigl[ f_{k}^{(i)}
f_{k+2}^{(i)} \bigr]
\\
&&\qquad =  \sum_{j \neq i} p_j
(1-p_j) \frac{1}{2} \bigl( 2 p_i
f_k^{(i)} D_k^{(i,j)} - p_i
f_{k+1}^{(i)} E_k^{(i,j)} + 2
(1-p_i) f_{k+1}^{(i)} D_k^{(i,j)}
\\
&&\hspace*{249pt}{}- (1-p_i) f_k^{(i)} E_{k+1}^{(i,j)}
\bigr)
\\
&&\qquad =:  \sum_{j \neq i} p_j
(1-p_j) b_{i,j},
\end{eqnarray*}
where each term in square brackets is expanded using Lemma~\ref{lem:soiBvar}.
Further by writing $f_k^{(i)} = (1-p_j) f_{k}^{(i,j)} + p_j
f_{k-1}^{(i,j)}$, we can rearrange the expression for $b_{i,j}$ as
%
\begin{eqnarray}
\label{eq:bdef} b_{i,j} & = & \tfrac{1}{2} p_i
p_j \bigl( \bigl( f^{(i,j)}_k \bigr)
^2 f^{(i,j)}_{k-1} - 2 \bigl( f^{(i,j)}_{k-1}
\bigr)^2 f^{(i,j)}_{k+1} + f^{(i,j)}_{k-2}
f^{(i,j)}_{k} f^{(i,j)}_{k+1} \bigr)
\nonumber
\\
& &{} + \tfrac{1}{2} (1-p_i) (1- p_j) \bigl(
\bigl( f^{(i,j)}_k \bigr) ^2 f^{(i,j)}_{k+1}
- 2 \bigl( f^{(i,j)}_{k+1} \bigr)^2
f^{(i,j)}_{k-1} + f^{(i,j)}_{k+2}
f^{(i,j)}_{k} f^{(i,j)}_{k-1} \bigr)
\nonumber
\\[-8pt]
\\[-8pt]
\nonumber
& &{} + \tfrac{1}{2} (1-p_i) p_j \bigl( 2 \bigl(
f^{(i,j)}_k \bigr) ^3 - 3 f^{(i,j)}_{k-1}
f^{(i,j)}_k f^{(i,j)}_{k+1} + \bigl(
f^{(i,j)}_{k-1} \bigr)^2 f^{(i,j)}_{k+2}
\bigr)
\\
& &{} + \tfrac{1}{2} p_i (1-p_j) \bigl( 2 \bigl(
f^{(i,j)}_k \bigr) ^3 - 3 f^{(i,j)}_{k-1}
f^{(i,j)}_k f^{(i,j)}_{k+1} + \bigl(
f^{(i,j)}_{k+1} \bigr)^2 f^{(i,j)}_{k-2}
\bigr).\nonumber
\end{eqnarray}
Similarly, using simplifications such as the fact that
\[
D_k = (1-p_i)^2 D_{k-1}^{(i)}
+ p_i^2 D_{k}^{(i)} +
p_i (1-p_i) E_k^{(i)}
\]
the second bracket of equation (\ref{eq:toprove}) can be written as
$p_i (1-p_i) p_j(1-p_j) c_{i,j}$, where
%
\begin{eqnarray}
\label{eq:cdef} c_{i,j} & := & \bigl( f^{(i,j)}_{k+1}
E_k^{(i,j)} - f^{(i,j)}_k
D^{(i,j)}_k - f^{(i,j)}_{k+2}
D^{(i,j)}_{k-1} \bigr)
\nonumber
\\
& = & - \bigl( f^{(i,j)}_k \bigr) ^3 + 2
f^{(i,j)}_{k-1} f^{(i,j)}_k
f^{(i,j)}_{k+1} - \bigl( f^{(i,j)}_{k+1}
\bigr)^2 f^{(i,j)}_{k-2}
\\
& &{} - \bigl( f^{(i,j)}_{k-1} \bigr)^2
f^{(i,j)}_{k+2} + f^{(i,j)}_{k-2}
f^{(i,j)}_k f^{(i,j)}_{k+2}.\nonumber
\end{eqnarray}
\upqed\end{pf}
%

\begin{lemma} \label{lem:checkcond4}
For the monotone Shepp--Olkin interpolation,
for each $i \neq j$, the
$b_{i,j} \geq0$ and
%
\begin{equation}
\label{eq:squarable} b_{i,j} \geq- \frac{p_i (1-p_j) + p_j(1-p_i)}{2} c_{i,j},
\end{equation}
and hence Condition \ref{cond:delta} is satisfied, and so the entropy
is concave.
\end{lemma}

\begin{pf} To verify the positivity of $b_{i,j}$, we simply observe
that each of the brackets
in equation (\ref{eq:bdef}) is positive, by an application of equations
(\ref{eq:cubic1}),
(\ref{eq:cubic2}), (\ref{eq:cubic3}) and (\ref{eq:cubic4}) proved in the \hyperref[sec:technical]{Appendix} below.

To verify that $b_{i,j} + \frac{1}{2}  (p_i (1-p_j) + p_j(1-p_i)
 )
c_{i,j}$ is positive, we consider
adding the final two terms of equation (\ref{eq:bdef}) to the
expression given in (\ref{eq:cdef}),
to obtain
\begin{eqnarray*}
&& (1-p_i) p_j \bigl( \bigl(
f^{(i,j)}_k \bigr)^3 - f^{(i,j)}_{k-1}
f^{(i,j)}_k f^{(i,j)}_{k+1} - \bigl(
f^{(i,j)}_{k+1} \bigr)^2 f^{(i,j)}_{k-2}
+ f^{(i,j)}_{k-2} f^{(i,j)}_k
f^{(i,j)}_{k+2} \bigr)
\\
& &\qquad{} + p_i(1-p_j)\\
&&\qquad\quad{}\times \bigl( \bigl( f^{(i,j)}_k
\bigr)^3 - f^{(i,j)}_{k-1} f^{(i,j)}_k
f^{(i,j)}_{k+1} - \bigl( f^{(i,j)}_{k-1}
\bigr)^2 f^{(i,j)}_{k+2} + f^{(i,j)}_{k-2}
f^{(i,j)}_k f^{(i,j)}_{k+2} \bigr),
\end{eqnarray*}
where the positivity of the two final terms is guaranteed by equations
(\ref{eq:cubic5})
and (\ref{eq:cubic6}) below.

Condition \ref{cond:delta} is verified by considering two separate
cases. If $c_{i,j} \geq0$, then all the terms in
equation (\ref{eq:diffexp}) are positive. Otherwise, if $c_{i,j} \leq
0$, then the
bracketed term in equation (\ref{eq:diffexp}) has negative discriminant
as a function of $p_i'$ and $p_j'$,
\begin{eqnarray*}
&&4 p_i (1-p_i) p_j (1-p_j)
c_{i,j}^2 - 4 b_{i,j} b_{j,i}\\
&&\qquad \leq
c_{i,j}^2 \bigl( 4 p_i (1-p_i)
p_j (1-p_j) - \bigl(p_i
(1-p_j) + p_j(1-p_i)\bigr)^2
\bigr)
\\
& &\qquad=  - c_{i,j}^2 (p_i - p_j)^2
\leq0,
\end{eqnarray*}
since under this assumption both sides of equation (\ref{eq:squarable})
are positive, so it can be squared.
In either case, we conclude that equation (\ref{eq:diffexp}) is
positive, and Condition~\ref{cond:delta} is satisfied.
\end{pf}
Since Condition \ref{cond:delta} has been verified, the proof of
Theorem~\ref{thm:SObis} is complete.

\begin{appendix}\label{app}
\section*{Appendix: Technical results regarding Bernoulli sums} \label{sec:technical}

In this section, we prove some technical results regarding the mass
functions of Bernoulli sum
random variables, required to prove the monotone Shepp--Olkin Theorem~\ref{thm:SObis}.

\begin{lemmaa} \label{lem:soiBvar}
Let $(f_k)_{k \in\Z}$ be the Bernoulli sum of parameters $p_1,\ldots,
p_m$. Then for every $k \in\Z$ and $q \geq1$, we have the identity
%
\begin{equation}
\label{eq:soiBvar} qf_k f_{k-q} = \sum
_{j=1}^m p_j(1-p_j)
\bigl[f^{(j)}_{k-1} f^{(j)}_{k-q}
-f^{(j)}_{k}f^{(j)}_{k-q-1} \bigr].
\end{equation}
\end{lemmaa}

\begin{pf}
We use induction on the number $m$ of parameters, the case where $m=1$
being obvious. If $(f_k)_{k \in\Z}$ is the Bernoulli sum of parameters
$p_1,\ldots, p_m$, we set for $p \in[0,1]$
\[
\tilde{f}_k := (1-p) f_k + p f_{k-1},
\]
and we want to prove, given $k \in\Z$ and $q \geq1$, that
%
\begin{eqnarray}
\label{eq:soiBvarRecu} q\tilde{f}_k \tilde{f}_{k-q}& =& \sum
_{j=1}^m p_j(1-p_j)
\bigl[\tilde {f}^{(j)}_{k-1} \tilde{f}^{(j)}_{k-q}
-\tilde{f}^{(j)}_{k}\tilde {f}^{(j)}_{k-q-1}
\bigr]
\nonumber
\\[-8pt]
\\[-8pt]
\nonumber
&&{} +p(1-p) [f_{k-1}f_{k-q}-f_k f_{k-q-1}
].
\end{eqnarray}
Expanding each side with the respect to the basis $(1-p)^2,2p(1-p), p^2$,
using the fact that $\tilde{f}_k^{(j)} := (1-p) f_k^{(j)} + p f_{k-1}^{(j)}$,
it is easy to show that equation \eqref{eq:soiBvarRecu} is satisfied
for some $k\in\Z$ and $q\geq1$ if \eqref{eq:soiBvar} holds true for
the pairs $(k,q)$,$(k,q+1)$ and $(k-1,q)$.
\end{pf}

Next
we prove the following technical inequality, which may be of
independent interest:

\begin{theoremm}\label{th:Symmetric}
$\operatorname{Property}(m)$ holds: that is,
for every $(g^{[m]}_k)_{k \in\Z}$ which is the probability mass function
of a sum of $m$
independent Bernoulli variables
%
\begin{eqnarray}
\label{eq:cubic1}\quad C_1(k) := g^{[m]}_{k-1} \bigl(
g^{[m]}_{k} \bigr)^2-2 \bigl(
g^{[m]}_{k-1} \bigr)^2 g^{[m]}_{k+1}+
g^{[m]}_{k} g^{[m]}_{k+1}
g^{[m]}_{k-2} \geq0\qquad
\mbox{for all $k \in\Z$.}\nonumber\\
\end{eqnarray}
\end{theoremm}

We first show that $\operatorname{Property}(m)$ implies a number of related
inequalities,
which are of
use elsewhere in the paper:

\begin{corollaryy} \label{cor:cubic}
If $\operatorname{Property}(m)$ holds, then for any $g^{[m]}$,
the probability mass function for the sum of $m$ independent Bernoulli
random variables, for all $k \in\Z$,
%
\begin{eqnarray}\qquad
\ol{C}_1(k)& :=& \bigl( g^{[m]}_k
\bigr)^2 g^{[m]}_{k+1} - 2 \bigl(g^{[m]}_{k+1}
\bigr)^2 g^{[m]}_{k-1} + g^{[m]}_{k+2}
g^{[m]}_{k} g^{[m]}_{k-1} \geq 0,
\label
{eq:cubic2}
\\
C_2(k)& := &\bigl( g^{[m]}_k \bigr)^3
- g^{[m]}_{k-1} g^{[m]}_k
g^{[m]}_{k+1} - \bigl( g^{[m]}_{k-1}
\bigr)^2 g^{[m]}_{k+2} + g^{[m]}_{k-2}
g^{[m]}_k g^{[m]}_{k+2}  \geq 0,
\label{eq:cubic5}
\\
\ol{C}_2(k) &:=& \bigl( g^{[m]}_k
\bigr)^3 - g^{[m]}_{k-1} g^{[m]}_k
g^{[m]} _{k+1} - \bigl( g^{[m]}_{k+1}
\bigr)^2 g^{[m]}_{k-2} + g^{[m]}_{k-2}
g^{[m]}_k g^{[m]}_{k+2}  \geq 0,
\label{eq:cubic6}
\\
C_3(k) &:=& 2 \bigl(g^{[m]}_k
\bigr)^3 - 3 g^{[m]}_{k-1} g^{[m]}_k
g^{[m]}_{k+1} + \bigl( g^{[m]}_{k-1}
\bigr)^2 g^{[m]}_{k+2}  \geq 0, \label{eq:cubic3}
\\
\ol{C}_3(k) &:=& 2 \bigl(g^{[m]}_k
\bigr)^3 - 3 g^{[m]}_{k-1} g^{[m]}_k
g^{[m]} _{k+1} + \bigl(g^{[m]}_{k+1}
\bigr)^2 g^{[m]}_{k-2} \geq 0, \label{eq:cubic4}
\\
D_1(k) &:=& 2 \bigl( g^{[m]}_k
\bigr)^2 g^{[m]}_{k-2} - 3 g^{[m]}_{k-2}
g^{[m]} _{k-1} g^{[m]}_{k+1} +
g^{[m]}_{k+1} g^{[m]}_{k}
g^{[m]}_{k-3} \geq 0. \label{eq:extra}
\end{eqnarray}
\end{corollaryy}

\begin{pf} First note that these inequalities can be paired up by a
duality argument.
That is, if $\operatorname{Property}(m)$ holds for every Bernoulli sum
$g^{[m]}$, it is
true for $\overline{g}_k := g^{[m]}_{m-k}$ with parameters $1-p_1, 1-p_2,
\ldots, 1-p_m$, which implies
equation~(\ref{eq:cubic2}). Similarly (\ref{eq:cubic5}) implies (\ref
{eq:cubic6}) and
(\ref{eq:cubic3}) implies (\ref{eq:cubic4}).
We write
\[
D^{[m]}_k = \bigl( g^{[m]}_k
\bigr)^2 - g^{[m]}_{k-1} g^{[m]}_{k+1},
\]
which is positive because $g^{[m]}$ is a sum of independent Bernoulli
random variables, and therefore log-concave.
In this notation, equations (\ref{eq:cubic5}), (\ref{eq:cubic3}) and
(\ref{eq:extra})
are a consequence of (\ref{eq:cubic1}), since simple calculations show that
\begin{eqnarray*}
g^{[m]}_{k+1} C_2(k) & = & 2\bigl(
g^{[m]}_k g^{[m]}_{k+1} -
g^{[m]}_{k-1} g^{[m]}_{k+2}\bigr)
D^{[m]}_k + g^{[m]}_{k+2}
C_1(k) \geq0,
\\
g^{[m]}_k C_3(k) & = & 2 \bigl(D^{[m]}_k
\bigr)^2 + g^{[m]}_{k-1} \ol{C}_1(k)
\geq0,
\\
g^{[m]}_{k-1} D_1(k) & = & 2 g^{[m]}_{k-2}
C_1(k) + g^{[m]}_{k+1} C_1(k-1) \geq0.
\end{eqnarray*}
Here, positivity of $( g^{[m]}_k g^{[m]}_{k+1} - g^{[m]}_{k-1}
g^{[m]}_{k+2}) $ is
again a consequence of log-concavity of $g^{[m]}$.
\end{pf}
In a similar way, we can argue that:

\begin{propositionn} \label{prop:PnEquiv1General}
If $\operatorname{Property}(m)$ is satisfied, then for every sum of $m$ independent
Bernoulli variables we have
%
\begin{equation}
\label{eq:PnEquiv1General} g^{[m]}_k D^{[m]}_{k}+
g^{[m]}_{k-2} D^{[m]}_{k+1} \geq2
g^{[m]}_{k+2} D^{[m]}_{k-1} \qquad\mbox{for every
$k \in\Z$.}
\end{equation}
\end{propositionn}

\begin{pf} We can restate $\operatorname{Property}(m)$ as being equivalent
to the inequality
\[
D^{[m]}_{k} \geq\frac{g^{[m]}_{k+1}}{g^{[m]}_{k-1}} D^{[m]}_{k-1},
\]
the iteration of which gives
\[
D^{[m]}_{k+1} \geq\frac{g^{[m]}_{k+2}}{g^{[m]}_{k} }\frac{
g^{[m]}_{k+1}}{ g^{[m]}_{k-1}}
D^{[m]}_{k-1},
\]
so equation \eqref{eq:PnEquiv1General} holds if we have
\[
\frac{g^{[m]}_{k} g^{[m]}_{k+1}}{g^{[m]}_{k-1} g^{[m]}_{k+2}} + \frac
{g^{[m]}_{k-2}
g^{[m]}
_{k+1}}{ g^{[m]}_{k} g^{[m]}_{k-1}} -2 \geq0,
\]
which can be rewritten as
\[
\frac{ C_1(k+1)}{g^{[m]}_{k-1} g^{[m]}_{k+1} g^{[m]}_{k+2}} + 2 \frac
{  (
D^{[m]}
_k  )^2}{
  ( g^{[m]}_k  )^2 g^{[m]}_{k-1} g^{[m]}_{k+1}} + \frac{ C_1(k)}{  (g^{[m]}_{k}  )^2 g^{[m]}_{k-1} } \geq0,
\]
which is positive by assumption, which proves the proposition.
\end{pf}

\begin{pf*}{Proof of Theorem~\ref{th:Symmetric}}
We prove $\operatorname{Property}(m)$ by induction on the number of
parameters $m$.
It is clear that $\operatorname{Property}(1)$ is true. Let us suppose
$\operatorname{Property}(m)$ holds for some
$m \geq1$. In order to prove
$\operatorname{Property}(m+1)$, it suffices to show that, for every $k \in
\Z$,
%
\begin{equation}
\label{eq:ToProve} g^{[m+1]}_{k-1} \bigl( g^{[m+1]}_{k}
\bigr)^2-2 \bigl( g^{[m+1]}_{k-1}
\bigr)^2 g^{[m+1]}_{k+1} + g^{[m+1]}_{k-2}
g^{[m+1]}_{k} g^{[m+1]}_{k+1} \geq0,
\end{equation}
where $g^{[m+1]}$ is the distribution of a sum of $m+1$ Bernoulli variables.
For $p = p_{m+1} \in[0,1]$, we can write
$ g^{[m+1]}_{k} := (1-p) g^{[m]}_{k}+p g^{[m]}_{k-1}$.
To prove that~\eqref{eq:ToProve} is positive, we expand it as a
polynomial in $p$, of order $3$, in the basis
$(1-p)^3,p(1-p)^2,p^2(1-p),p^3$, and show that the coefficients of each
of these terms are positive:
\begin{longlist}[(1)]
\item[(1)]
The coefficient of $(1-p)^3$ is $C_1(k)$, which is clearly positive, by
$\operatorname{Property}(m)$.

\item[(2)]
The coefficient of $p^3$ is $C_1(k-1)$,
which is also positive, by $\operatorname{Property}(m)$.

\item[(3)]
The coefficient of $p(1-p)^2$ is $D_1(k)$, which is positive, since
$\operatorname{Property}(m)$ implies
equation (\ref{eq:extra}).

\item[(4)]
The coefficient of $p^2(1-p)$ can be written
$g^{[m]}_{k-1} D^{[m]}_{k-1} + g^{[m]}_{k-3} D^{[m]}_{k} - 2
g^{[m]}_{k+1} D^{[m]}
_{k-2}$, which is positive by Proposition~\ref{prop:PnEquiv1General}.
\end{longlist}
Since each coefficient is positive, we deduce that
equation \eqref{eq:ToProve} is satisfied, which shows that $\operatorname{Property}(m+1)$ holds.
\end{pf*}
\end{appendix}

\section*{Acknowledgment}
The authors would like to thank an anonymous referee for a careful
reading of this paper, which has resulted in numerous improvements to it.


%

%




\printaddresses
\end{document}